\documentclass[11pt]{article}
\usepackage[colorlinks=true]{hyperref}
\usepackage{amsmath,amsfonts,amssymb,amsthm,mathrsfs}
\usepackage[a4paper,vmargin={3.5cm,3.5cm},hmargin={2.5cm,2.5cm}]{geometry}
\usepackage[font=sf, labelfont={sf,bf}, margin=1cm]{caption}
\usepackage{graphicx,graphics}
\usepackage{epsfig}
\usepackage{latexsym}
\usepackage[applemac]{inputenc}
\usepackage{ae,aecompl}
\usepackage[english]{babel}
\usepackage[toc,page]{appendix}
\usepackage{bbm}

\date{}

\newcommand{\E}{\mathbb{E}}

\newcommand{\F}{\ensuremath{\mathcal{F}}}
\newcommand{\G}{\ensuremath{\mathcal{G}}}

\renewcommand{\subset}{\subseteq}

\newcommand{\e}{{\mathrm e}}
\newcommand{\R}{{\mathbb{R}}}
\newcommand{\T}{\mathcal{T}}
\newcommand{\TT}{\mathbb{T}}
\newcommand{\s}{\overline{\mathcal{S}}^{\downarrow}}
\newcommand{\p}{\overline{\mathcal{P}}}
\newcommand{\N}{\mathbb{N}}

\newcommand{\Z}{\mathbb{Z}}

\newcommand{\veps}{\varepsilon}

\newcommand{\pr}{\mathbb{P}}

\renewcommand{\geq}{\geqslant}
\renewcommand{\leq}{\leqslant}

\def\llbracket{[\hspace{-.10em} [ }
\def\rrbracket{ ] \hspace{-.10em}]}

\newtheorem{thm}{Theorem}[section]
\newtheorem{defn}[thm]{Definition}
\newtheorem{lem}[thm]{Lemma}
\newtheorem{prop}[thm]{Proposition}
\newtheorem{cor}[thm]{Corollary}

\newtheorem{rem}[thm]{Remark}
\newtheorem{ex}[thm]{Example}
\date{}

\title{\bf{\textsc{On the exponential functional of Markov Additive Processes, and \\ applications to multi-type self-similar fragmentation processes and trees.}}}

\author{Robin Stephenson\thanks{Department of Statistics, University of Oxford, 24-29 St Giles', Oxford OX1 3LB, UK, robin.stephenson@normalesup.org }}

\begin{document}

\maketitle

\begin{abstract}
Markov Additive Processes are bi-variate Markov processes of the form $(\xi,J)=\big((\xi_t,J_t),t\geq0\big)$ which should be thought of as a multi-type L\'evy process: the second component $J$ is a Markov chain on a finite space $\{1,\ldots,K\}$, and the first component $\xi$ behaves locally as a L\'evy process with dynamics depending on $J$. In the subordinator-like case where $\xi$ is nondecreasing, we establish several results concerning the moments of $\xi$ and of its \emph{exponential functional} $I_{\xi}=\int_{0}^{\infty} e^{-\xi_t}\mathrm dt,$ extending the work of Carmona et al. \cite{CPY97}, and Bertoin and Yor \cite{BY01}.

We then apply these results to the study of multi-type self-similar fragmentation processes: these are self-similar transformations of Bertoin's homogeneous multi-type fragmentation processes, introduced in \cite{Ber08}. Notably, we encode the genealogy of the process in an $\R$-tree as in \cite{HM04}, and under some Malthusian hypotheses, compute its Hausdorff dimension in a generalisation of our previous results in \cite{Steph13}.
\end{abstract}

\section*{Introduction}
A \emph{Markov Additive Process} $(\xi,J)=\big((\xi_t,J_t),t\geq0\big)$ is a (possibly killed) Markov process on $\R\times\{1,\ldots,K\}$ for some $K\in\N$ such that, calling $\pr_{x,i}$ its distribution starting from some point $(x,i)\in \R\times\{1,\ldots,K\}$, we have for all $t\geq 0$
\[
\text{ under } \mathbb P_{(x,i),}\;\left(\left(\xi_{t+s}-\xi_t,J_{t+s}),s \geq 0\right) \ | \ (\xi_u,J_u),u\leq t \right)  \text{ has distribution } \mathbb P_{(0,J_t)}.
\]
MAPs should be thought of as multi-type L\'evy processes, whose local dynamics depend on an additional discrete variable. 

In this paper, we focus on the case where the position component $\xi$ is nonincreasing, and we are interested in computing various moments of variables related to $(\xi,J)$. Most importantly, we study the so-called \emph{exponential functional} 
\[I_{\xi}=\int_0^{\infty} e^{-\xi_t}\mathrm dt.
\]
In the classical one-type case (not always restricted to the case where $\xi$ is nonincreasing), motivations for studying the exponential functional stem from mathematical finance, self-similar Markov processes, random processes in random environment, and more, see the survey paper \cite{BY05}. Here in the multi-type setting, we are most of all interested in the power moments of $I_{\xi}$, see Propositions \ref{yu} and \ref{negativemoments}. This generalises results of Carmona, Petit and Yor \cite{CPY97} for the positive (and exponential) moments, and Bertoin and Yor \cite{BY01} for the negative moments.

\medskip

Our main interest in MAPs here lies in their applications to fragmentation processes. Such processes describes the evolution of an object which continuously splits in smaller fragments, in a branching manner. Several kinds of fragmentation processes have been studied, notably by Jean Bertoin, who introduced the \emph{homogeneous}, \emph{self-similar} and \emph{homogeneous multi-type} kinds in respectively \cite{BertoinHomogeneous}, \cite{BertoinSSF}, \cite{Ber08}. Motivations for studying multi-type cases stem from the fact that, in some physical processes, particles can not be completely characterised by their mass alone, and we need some additional information such as their shape, or their environment. See also \cite{Norris2000} for a model of multi-type \emph{coagulation}.

We look here at fragmentations which are both multi-type and self-similar: this means that, on one hand, the local evolution of a fragment depends on its type, which is an integer in $\{1,\ldots,K\}$, and that a fragment with size $x\in(0,1]$ evolves $x^{\alpha}$ times as fast as a fragment with size $1$, where $\alpha\in\R$ is a parameter called the index of self-similarity.

Many pre-existing results which exist for self-similar fragmentations with only one type have counterparts in this multi-type setting. Of central importance is Bertoin's characterisation of the distribution of a fragmentation via three sets of parameters. Additionally to the index of self-similarity $\alpha$, there are $K$ \emph{dislocation measures} $(\nu_i,i\in\{1,\ldots,K\})$, which are $\sigma$-finite measures on the set $\s$ of $K$-type partitions of $1$ (an element of this set can be written as $\bar{\mathbf{s}}=(s_n,i_n)_{n\in\N},$ where $(s_n)_{n\in\N}$ is a nonincreasing sequence of nonnegative numbers adding to at most one, while $(i_n)_{n\in\N}$ gives a type to each fragment $s_n$ with $s_n=0$, see Section \ref{partitionsdef} for a precise definition) which satisfy some integrability conditions. These encode the splittings of particles, in the sense that a particle with mass $x$ and type $i$ will, informally, split into a set of particles with masses $(xs_n,n\in\N)$ and types $(i_n,n\in\N)$ at rate $x^{\alpha}\mathrm d\nu_i(\bar{\mathbf{s}}).$ Moreover, there are also $K$ \emph{erosion rates} $(c_i,i\in\{1,\ldots,K\})$ which encode a continuous, deterministic shaving of the fragments.

Amongst other results which generalise from the classical to multi-type setting is the appearance of \emph{dust}: when $\alpha<0$, even if there is no erosion and each individual splitting preserves total mass, we observe that this total mass decreases and the initial object is completely reduced to zero mass in finite time. This phenomenon was first observed by Filippov (\cite{Filippov}) in a slightly different setting, and then in the classical self-similar fragmentation setting by Bertoin \cite{BertoinAB}. Here we will extend a result of \cite{H03} to establish that the time at which all the mass has disappeared has some finite exponential moments. Using this, we then to show that the genealogy of the fragmentation can be encoded in a compact continuum random tree, called \emph{multi-type fragmentation tree}, as in \cite{HM04} and \cite{Steph13}. One important application of these trees will be found in our upcoming work \cite{HS15}, where we will show that they naturally appear as the scaling limits of various sequences of discrete trees.

An interesting subclass of fragmentations is those which are called \emph{Malthusian}. A fragmentation process is called Malthusian if there exists a number $p^*\in[0,1]$ called the \emph{Malthusian exponent} such that the $K\times K$ matrix whose $(i,j)$-th entry is 
\[p^*c_i\mathbbm{1}_{\{i=j\}}+\left(\int_{\s}\left(\mathbbm{1}_{\{i=j\}}-\sum_{n=1}^{\infty}s_n^{p^*}\mathbbm{1}_{\{i_n=j\}}\right)\nu_i(\mathrm d \bar{\mathbf{s}})\right)\]
has $0$ as its smallest real eigenvalue. This is implies that, as shown in Section \ref{sec:malt}, if $\alpha=0$, there exists positive numbers $(b_1,\ldots,b_K)$ such that, calling $\big(X_n(t),n\in\N\big)$ the sizes of the particles of the fragmentation process at time $t$, and $\big(i_n(t),n\in\N\big)$ their respective types, the process
\[\Big(\sum_{n\in\N} b_{i_n(t)}X_n(t)^{p^*},t\geq 0\Big)\]
is a martingale (in fact a generalisation of the classical \emph{additive martingale} of branching random walks). In particular, if the system is conservative in the sense that there is no erosion and each splitting preserves total mass, then, as in the one-type case, we have $p^*=1$. In the Malthusian setting, the additive martingale can then be used to study the fragmentation tree in more detail, culminating with Theorem \ref{dimension}: under a slightly stronger Malthusian assumption, either the set of leaves of the fragmentation tree is countable, or its Hausdorff dimension is equal to $\frac{p^*}{|\alpha|}.$

\medskip

The paper is organised as follows. In Sections $1$ to $3$ we introduce and study respectively MAPs and their exponential functionals, multi-type fragmentation processes, and multi-type fragmentation trees. At the end, Section $4$ focuses on the Hausdorff dimension of the leaves of the fragmentation tree: Theorem \ref{dimension} and its proof.

\medskip

\textit{An important remark}: several of the results presented here are generalisations of known results for the monotype case which were obtained in previous papers (in particular \cite{BertoinHomogeneous},\cite{BertoinSSF},\cite{Ber08},\cite{H03}\cite{HM04}, and \cite{Steph13}). At times, the proofs of the generalised results do not differ from the originals in a significant manner, in which case we might not give them in full detail and instead refer the reader to the original papers. However, we also point out that our work is not simply a straightforward generalisation of previous results, and the multi-type approach adds a linear algebra dimension to the topic which is interesting in and of itself.
\medskip

\textit{Some points of notation:} $\N$ is the set of positive integers $\{1,2,3,\ldots,\}$, while $\Z_+$ is the set of nonnegative integers $\N\cup\{0\}.$ Throughout the paper, $K\in\N$ is fixed and is the number of types of the studied processes. We use the notation $[K]=\{1,\ldots,K\}$ for the set of types.

Vectors in $\R^K$, sometimes interpreted as row matrices and sometimes as column matrices, will be written in bold: $\mathbf{v}=(v_i)_{i\in[K]}$. $K\times K$ matrices will be written in capital bold: $\mathbf{A}=(A_{i,j})_{i,j\in[K]}.$ If a matrix does not have specific names for its entries, we put the indexes after bracketing the matrix, for example $(e^{\mathbf{A}})_{i,j}$ is the $(i,j)$-th entry of $e^{\mathbf{A}}.$ $\mathbf{1}$ is the column matrix with all entries equal to $1$, and $\mathbf{I}$ is the identity matrix.

If $X$ is a real-valued random variable and $A$ and event, we use $\E[X,A]$ to refer to $\E[X\mathbf{1}_{A}]$ in a convenient fashion. Moreover, we use the convention that $\infty \times 0 =0$, so in particular, $X$ being infinite outside of $A$ does not pose a problem for the above expectation.
\section{Markov Additive Processes and their exponential functionals}
\subsection{Generalities on Markov additive processes}

We give here some background on Markov additive processes and refer to Asmussen \cite[Chapter XI]{asmussen} for details and other applications. 

\smallskip

\begin{defn} Let  $\left((\xi_t,J_t),t \geq 0 \right)$ be a  Markov process on $\mathbb R \times \{1,\ldots,K\}\cup\{(+\infty,0)\}$, where $K \in \mathbb N$, and write $\mathbb P_{(x,i)}$ for its distribution when starting at a point $(x,i)$. It is called a \emph{Markov additive process} \emph{(MAP)} if for all $t \in \mathbb R_+$ and all $(x,i) \in \mathbb R \times \{1,\ldots,K\}$,
\[
\text{ under } \mathbb P_{(x,i),}\;\left(\left(\xi_{t+s}-\xi_t,J_{t+s}),s \geq 0\right) \ | \ (\xi_u,J_u),u\leq t , \xi_t<\infty \right)  \text{ has distribution } \mathbb P_{(0,K_t)},
\]
and $(+\infty,0)$ is an absorbing state.
\end{defn}

\smallskip

MAPs can be interpreted as multi-type L\'evy processes: when $K=1$, $\xi$ is simply a standard L\'evy process, while in the general case, $(J_t,t\geq 0)$ is a continuous-time Markov chain, and on its constancy intervals, the process $\xi$ behaves as a L\'evy process, whose dynamics depend only on the value of $J$. Jumps of $J$ may also induce jumps of $\xi$. In this paper, we always consider MAPs such that $\xi$ is \textbf{non-decreasing}, that is, the MAP analogue of subordinators. The distribution of such a process is then characterised by three groups of parameters:
\begin{enumerate}
\item[$\bullet$] the transition rate matrix $\mathbf{\Lambda}=(\lambda_{i,j})_{i,j\in[K]}$ of the Markov chain $(J_t,t\geq 0)$.
\item[$\bullet$] a family $(B_{i,j})_{i,j\in[K]}$ of probability distributions on $[0,+\infty)$: for $i\neq j,$ $B_{i,j}$ is the distribution of the jump of $\xi$ when $J$ jumps from $i$ to $j$. If $i=j$, we let $B_{i,i}$ be the Dirac mass at $0$ by convention. We also let $\widehat{B}_{i,j}(p)=\int_0^{\infty}e^{-px}B_{i,j}(\mathrm dx).$
\item[$\bullet$] triplets $(k^{(i)}, c^{(i)},\Pi^{(i)})$, where, for each $i\in[K]$, $k^{(i)}\geq0,c^{(i)}\geq 0$ and $\Pi^{(i)}$ is a $\sigma$-finite measure on $(0,\infty)$ such that $\int_{(0,\infty)} (1 \wedge x) \Pi^{(i)}(\mathrm dx)<\infty$.
The triplet $(k^{(i)}, c^{(i)},\Pi^{(i)})$ corresponds to the standard parameters (killing rate, drift and L\'evy measure) of the subordinator which $\xi$ follows on the time intervals where $J=i$. We call $(\psi_i)_{i\in\{1,\ldots,K\}}$ the corresponding Laplace exponents, that is, for $i\in[K]$, $p\geq 0$
\[\psi^{(i)}(p)=k^{(i)}+c^{(i)}p+\int_0^{\infty} (1-\e^{-px})\Pi^{(i)}(\mathrm{d}x).\]
\end{enumerate}

All these parameters can then be summarised in a generalised version of the Laplace exponent for the MAP, which we call the Bernstein matrix $\mathbf{\Phi}(p)$ for $p\geq 0$, which is a $K\times K$ matrix defined by
\begin{equation}\label{expressionBernstein}
\mathbf{\Phi}(p)=\Big(\psi_i(p)\Big)_{\mathrm{diag}} -\mathbf{\Lambda}\circ\mathbf{\widehat{B}}(p). 
\end{equation}
Here $\circ$ denotes the entrywise product of matrices, and $\mathbf{\widehat{B}}(p)=\Big(\widehat{B}_{i,j\in[K]}(p)\Big)_{i,j}.$ We then have, for all $t\geq 0$, $p\geq 0$ and all types $i,j$, by Proposition 2.2 in \cite[Chapter XI]{asmussen},
\begin{equation}\label{eq:bernstein}
\mathbb{E}_i\big[e^{-p\xi_t},J_t=j\big]=\Big(e^{-t\mathbf{\Phi}(p)}\Big)_{i,j}.
\end{equation}

Note that this can be extended to negative $p$. Specifically, let

\begin{equation}\label{defunderline}
\underline{p}=\inf \Big\{p\in\R:\; \forall i, \int_0^{\infty}(\e^{-px}-1)\Pi^{(i)}(\mathrm{d}x)<\infty\text{ and } \forall i,j\in[K],\, \lambda_{i,j}\int_0^{\infty}e^{-px}B_{i,j}(\mathrm dx)<\infty.\Big\}
\end{equation}
Then, $\mathbf{\Phi}$ can be analytically extended to $(\underline{p},\infty),$ and then (\ref{eq:bernstein}) holds for $p>\underline{p}$. Note that, when considering (\ref{eq:bernstein}) with $p<0$, the restriction to the event $\{J_t=j\}$ for $j\in[K]$ precludes killing, thus $e^{-p\xi_t}$ cannot be infinite.

We will always assume {\bf that the Markov chain of types is irreducible, and that the position component isn't a.s. constant} (that is, one of the Laplace exponents $\psi_i$ is not trivial, or one of the $B_{i,j}$ charges $(0,\infty)$).
\subsection{Some linear algebra}
We give in this section some tools which will let us study the eigenvalues and eigenvectors of the Bernstein matrix of a MAP.
\begin{defn} We say that a matrix $\mathbf{A}=(A_{i,j})_{i,j\in[K]}$ is an \emph{ML-matrix} if its off-diagonal entries are all nonnegative. We then say that it is \emph{irreducible} if, for all types $i$ and $j$, there exists a sequence of types $i_1=i,i_2,\ldots,i_n=j$ such that $\prod_{k=1}^{n-1}A_{i_k,i_{k+1}}>0$.
\end{defn} 

\noindent Notice that, for all $p\geq 0$, $-\mathbf{\Phi}(p)$ is an ML-matrix. 
%If $\psi_i(0)=0$ for all $i$, then $-\mathbf{\Phi}(0)=\mathbf{\Lambda}$ is even a $\mathsf Q$--matrix (its lines all add up to zero).

The following proposition regroups most properties of ML-matrices which we will need. For an ML-matrix $\mathbf{A}$, we let $\lambda(\mathbf{A})$ be the maximal real part of the eigenvalues of $\mathbf{A}.$

\begin{prop}\label{algebre}
Let $\mathbf{A}$ and $\mathbf{B}$ be two ML-matrices, $\mathbf{A}$ being irreducible. Assume that $A_{i,j}\geq B_{i,j}$ for all $i,j$, and assume also that their there exists $k$ and $l$ such that $A_{k,l}> B_{k,l}.$ We then have the following:
\begin{itemize}
\item[$(i)$] $\lambda(\mathbf{A})$ is a simple eigenvalue of $\mathbf{A}$, and there is a corresponding eigenvector with strictly positive entries.
\item[$(ii)$] Any nonnegative eigenvector of $\mathbf{A}$ corresponds to the eigenvalue $\lambda(\mathbf{A})$.
\item[$(iii)$] For any eigenvalue $\mu$ of $\mathbf{A}$, we have $\mathrm{Re}(\mu) < \lambda(\mathbf{A})$.
\item[$(iv)$] $\lambda(\mathbf{A})$ is a continuous function of the entries of $\mathbf{A}$.
\item[$(v)$] For all $i$ and $j$, $(e^{\mathbf{A}})_{i,j}>(e^{\mathbf{B}})_{i,j}$.
\item[$(vi)$] $\lambda(\mathbf{A})>\lambda(\mathbf{B}).$
\end{itemize}
\end{prop}
Note that $(iv)$ implies that $e^{\mathbf{A}}$ only has strictly positive entries.
\begin{proof}

Points $(i),$ $(ii),$ $(iii)$ and $(iv)$ are classical for nonnegative matrices ($(i),$ $(ii),$ and $(iii)$ are just part of the Perron-Frobenius theorem, while an elementary proof of $(iv)$ can be found in \cite{Meyer}), and are readily generalised to any ML-matrix by adding a sufficiently large multiple of the identity matrix.

For $(v)$, take $x>0$ large enough so that both $x\mathbf{I}+\mathbf{A}$ and $x\mathbf{I}+\mathbf{B}$ are both non-negative. A trivial induction shows that $(x\mathbf{I}+\mathbf{A})^n_{i,j}\geq (x\mathbf{I}+\mathbf{B})^n_{i,j}$ for all $i,j$, implying by the series expression of the exponential that $e^x(e^\mathbf{A})_{i,j}\geq e^x(e^\mathbf{B})_{i,j}$. Moreover, by irreducibility of $\mathbf{A}$, we can chose $i_1,\ldots,i_n$ such that $i_1=i$, $i_n=j$, $i_m=k$ and $i_{m+1}=l$ for some $1\leq m\leq n-1$ and $A_{i_p,i_{p+1}}>0$ for all $1\leq m\leq n-1$. This implies $\big((x\mathbf{I}+\mathbf{A})^n\big)_{i,j}>\big((x\mathbf{I}+\mathbf{B})^n\big)_{i,j}$, hence $e^x(e^{\mathbf{A}})_{i,j}> e^x(e^{\mathbf{B}})_{i,j}.$

\medskip

%We prove both $(ii)$ and $(iii)$ at once. Let $\lambda$ be the maximal real part of the eigenvalues of $A$ (note that we do not know that it is an eigenvalue yet). Since the eigenvalues of $e^{\mathbf{A}}$ are the exponentials of those of $\mathbf{A}$, which implies that $e^\lambda$ is the spectral radius of $e^{\mathbf{A}}$. By the Perron-Frobenius theorem, $e^\lambda$ is a simple eigenvalue of $e^{\mathbf{A}}$, it has a corresponding eigenvector $\mathbf{v}$ with strictly positive entries, and all other eigenvalues have strictly lower modulus. Since proves $(iii)$. To finish $(ii)$, we have to prove that $\mathbf{A}\mathbf{v}=\lambda \mathbf{v}$. Standard algebra results imply the existence of a polynomial $P$ such that $P(e^{\mathbf{A}})=\mathbf{A}$ and $P(e^\lambda)=\lambda'$ where $\lambda'$ is the unique eigenvalue of $A$ with real part $\lambda$. Thus $\mathbf{A}v=P(e^{\mathbf{A}})\mathbf{v}=P(e^\lambda) \mathbf{v}=\lambda' \mathbf{v}$. Since $\mathbf{A}$ and $\mathbf{v}$ have real entries, this implies that $\lambda'$ is real, hence $\lambda'=\lambda$ and $\mathbf{A}v=\lambda v$.

%\medskip

To prove $(vi)$, we use the Collatz-Wielandt formula, see for example \cite[Exercise 1.6]{Seneta}, which, applied to $e^{\mathbf{A}}$, states that
\[e^{\lambda(\mathbf{A})}=\underset{\mathbf{v}\in\R_{\geq 0}^K}\sup \,\underset{i:v_i\neq 0}\inf \frac{(e^{\mathbf{A}}\mathbf{v})_i}{v_i}.\]
Taking $\mathbf{v}$ such that $\mathbf{B}\mathbf{v}=\lambda(\mathbf{B})\mathbf{v}$, we have by $(v)$ that $(e^{\mathbf{A}}\mathbf{v})_i>(e^{\mathbf{B}}\mathbf{v})_i=e^{\lambda(\mathbf{B})}v_i$ for all $i$ such that $v_i\neq 0$, implying $e^{\lambda(\mathbf{A})}>e^{\lambda(\mathbf{B})}.$
\end{proof}

\begin{cor}\label{inversible} For all $p>\underline{p}$ such that $-\lambda(-\mathbf{\Phi}(p))>0$, $\mathbf{\Phi}(p)$ is invertible. In particular, $\mathbf{\Phi}(p)$ is invertible for $p>0$, and $\mathbf{\Phi}(0)$ there is at least one $i\in[K]$ such that $k^{(i)}>0.$
\end{cor}

%\subsection{Natural martingales and biasing}

\subsection{Moments at the death time}
Assume that the MAP dies almost surely, that is $k^{(i)}>0$ for at least one $i\in[K]$. Let 
\[T=\inf\{t\geq 0:\, \xi_t=\infty\}\]
be the death time of $\xi$. Then, for $i\in[K]$, and $p\in\R,$ let
\[f_i(p)=\E_i[e^{-p\xi_{T^-}}].\] 
\begin{prop}\label{deathmoment}
Take $p>\underline{p}$ such that $-\lambda(-\mathbf{\Phi}(p))>0$. Let, for notational purposes, $\mathbf{F}(p)=(f_i(p),i\in[K])$ and $\mathbf{K}=(k^{(i)},i\in[K])$ in column matrix form. We then have
\[\mathbf{F}(p)=(\mathbf{\Phi}(p))^{-1}\mathbf{K}\]
\end{prop}

We start with a lemma which is essentially a one-type version of Proposition \ref{deathmoment}.
\begin{lem}\label{deathmoment1}
Let $(\xi_t,t\geq 0)$ be a non-killed subordinator with Laplace exponent $\psi: \R\to \R\cup\{-\infty\}$. Let $T$ be an independent exponential variable with parameter $k$, we then have, for $p\in\R$ such that $k+\psi(p)>0$. We then have
\[\E[e^{-p\xi_T}]=\frac{k}{k+\psi(p)}\]
\end{lem}
\begin{proof}
By independence, we can write
\[\E[e^{-p\xi_T}]=\int_{0}^{\infty}ke^{-kt}\E[e^{-p\xi_t}]\mathrm dt=\int_0^{\infty}ke^{-kt}e^{-t\psi(p)}\mathrm d t=\frac{k}{k+\psi(p)}.\]
\end{proof}

\noindent \textit{Proof of Proposition \ref{deathmoment}.} We start by considering $p\geq 0$ only. Let $\tau$ be the time of first type change of the MAP. We use the strong Markov property at time $\tau\wedge T$ and get 
\[f_i(p)=\E_i[e^{-p\xi_{T^-}},T\leq \tau] + \sum_{j\neq i}\E_i[e^{-p\xi_{\tau}},\tau<T,J_{\tau}=j]f_j(p)\]

Note that, until $\tau\wedge T$, $\xi$ behaves as a non-killed subordinator $\tilde{\xi}^{(i)}$ with Laplace exponent $\tilde{\psi}^{(i)}$ given by $\tilde{\psi}^{(i)}(p)=\psi^{(i)}(p)-k^{(i)}$, while $\tau$ and $T$ can be taken as two independent exponential variables with respective parameters $k^{(i)}$ and $|\lambda_{i,i}|$. Moreover, if jumping to type $j$ at time $\tau$, then there is a jump with distribution $B_{i,j}$. Hence we can write
\[f_i(p)= \pr_i[T\leq \tau]\E[e^{-\tilde{\xi}^{(i)}(T\wedge\tau)}]+\sum_{j\neq i}\pr_i[\tau <T]\pr_i[J_{\tau}=j]\E[e^{-p\tilde{\xi}^{(i)}_{\tau\wedge T}}]\hat{B}_{i,j}(p)f_j(p).\]
%Note that $\tilde{\psi}^{(i)}(p)+k^{(i)}+|\lambda_{i,i}|=\psi^{(i)}(p)-\lambda_{i,i}$ is the $i$-th diagonal term of $\mathbf{\Phi}(p)$. However, $-\lambda(-\mathbf{\Phi}(p))>0$, the smallest eigenvalue of $\mathbf{\Phi}(p)$, is positive and smaller than the diagonal coefficients (seen by taking an eigenvector), 
Since $p\geq0$, $\tilde{\psi}^{(i)}(p)+k^{(i)}+|\lambda_{i,i}|>0$ and we can apply Lemma \ref{deathmoment1}:
\begin{align*}
f_i(p)&=\frac{k^{(i)}}{|\lambda_{i,i}|+k^{(i)}}\frac{|\lambda_{i,i}|+k^{(i)}}{\tilde{\psi}^{(i)}(p)+|\lambda_{i,i}|+k^{(i)}}+\sum_{j\neq i} \frac{|\lambda_{i,i}|}{|\lambda_{i,i}|+k^{(i)}}\frac{\lambda_{i,j}}{|\lambda_{i,i}|}\frac{|\lambda_{i,i}|+k^{(i)}}{\tilde{\psi}^{(i)}(p)+|\lambda_{i,i}|+k^{(i)}}\hat{B}_{i,j}(p)f_j(p) \\
&=\frac{k^{(i)}}{\psi^{(i)}(p)+|\lambda_{i,i}|}+\sum_{j\neq i}\frac{\lambda_{i,j}}{\psi^{(i)}(p)+|\lambda_{i,i}|}\hat{B}_{i,j}(p)f_j(p).
\end{align*}
Recalling that $\lambda_{i,i}<0$, we have
\[\psi^{(i)}(p)f_i(p)=k^{(i)}+\sum_{j=1}^K \hat{B}_{i,j}(p)\lambda_{i,j}f_j(p).\]
This can be rewritten in matrix form as
\[\Big(\psi_i(p)\Big)_{\mathrm{diag}}\mathbf{F}(p)= \mathbf{K}+\big(\mathbf{\Lambda}\circ\mathbf{\widehat{B}}(p)\big)\mathbf{F}(p)\]
where we recall that $\circ$ indicates the entrywise product of matrices. Recalling the expression of $\mathbf{\Phi}(p)$ from \ref{expressionBernstein}, we then see that
\[\mathbf{\Phi}(p)\mathbf{F}(p)=\mathbf{K}.\]
And since $\mathbf{\Phi}(p)$ is invertible, we do end up with

\begin{equation}\label{partial}
\mathbf{F}(p)=(\mathbf{\Phi}(p))^{-1}\mathbf{K}.
\end{equation}

Now we want to extend this to negative $p$ such that $-\lambda(-\mathbf{\Phi}(p))>0$. Since the coefficients of $\mathbf{\Phi}(p)$ have an analytic continuation to $(\underline{p},\infty)$, those of $(\mathbf{\Phi}(p))^{-1}$ have such a continuation on the domain where $\mathbf{\Phi}(p)$ is invertible. By classical results, this implies that equation (\ref{partial}) extends to such $p$.
\qed
\subsection{The exponential functional}
We are interested in the random variable $I_{\xi}$ called the \emph{exponential functional} of $\xi$, defined by
\[I_{\xi}=\int_{0}^{\infty}e^{-\xi_t}\mathrm dt.\]
The fact that it is well-defined and finite a.s. is a consequence of this law of large numbers-like lemma.
\begin{lem} 
\label{lem:linearite}
As $t\to\infty$, the random variable $t^{-1}\xi_t$ has an almost--sure limit, which is strictly positive (and possibly infinite).
\end{lem}

\begin{proof} Note that, if any $k^{(i)}$ is nonzero, by irreducibility, the process will be killed a.s. and the wanted limit is $+\infty.$ We can thus assume that there is no killing. Let $i$ be any type for which $\psi_i$ is not trivial, or at least one $B_{i,j}$ gives positive mass to $(0,\infty)$. Let then $(T_n,n\in \N)$ be the successive return times to $i$. It follows from the definition of a MAP that $(T_n,n\in\N)$ and $(\xi(T_n),n\in\N)$ are both random walks on $\R$, in the sense that the sequences $(T_{n+1}-T_n,n\in\N)$ and $(\xi(T_{n+1})-\xi(T_n),n\in\N)$ are both i.i.d). For $t\geq 0$, we then let $n(t)$ be the unique integer such that $T_{n(t)}\leq t <T_{n(t)+1},$ and writing
\[\frac{\xi_{T_{n(t)}}}{T_{n(t)+1}}\leq\frac{\xi_t}{t}\leq \frac{\xi_{T_{n(t)+1}}}{T_{n(t)}},\]
we can see by the strong law of large numbers that both bounds converge to the same limit, ending the proof.
\end{proof}

We are interested in the power moments of $I_{\xi},$ which are most easily manipulated in column matrix form: for appropriate $p\in\R$, we let $\mathbf{N}(p)$ be the column vector such that 
\[\big(\mathbf{N}(p)\big)_i=\E_i[I_{\xi}^p]\] for all $i\in[K]$. We mention that some work on this has already been done, see notably Proposition 3.6 in \cite{KKPW14}.
\subsubsection{Positive and exponential moments}

\begin{prop}\label{yu} \quad 
\begin{itemize}
\item[$(i)$]For an integer $k\geq 0$, we have 
\begin{equation}\label{formulemomentpos}
\mathbf{N}(k)=k!\left(\, \prod_{l=0}^{k-1}\big({\mathbf{\Phi}(k-l)}\big)^{-1}\right)\mathbf{1}.
\end{equation}
\item[$(ii)$]For all $a<\rho\big(\underset{k\to\infty}\lim (\mathbf{\Phi}(k))^{-1}\big)$, (where $\rho$ denotes the spectral radius of a matrix), we have
\[\E_i[e^{aI_{\xi}}]<\infty\]
for all $i$.
\end{itemize}
\end{prop}

Equation (\ref{formulemomentpos}) is a consequence of the following recursive lemma.

\begin{lem}\label{positivemomentsrecursion}
We have, for $p\geq 1,$
\[\mathbf{N}(p)=p\big({\mathbf{\Phi}(p)}\big)^{-1}\mathbf{N}(p-1)\]
\end{lem}

\begin{proof} We combine the strategy used in \cite{BY05} with some matrix algebra. Let, for $t\geq 0$,
\[I_t=\int_{t}^{\infty}e^{-\xi_s}\mathrm ds.\]
By integrating the derivative of $I_t^p$, we get
\[I_0^p-I_1^p=p\int_{0}^1 e^{-\xi_s}I_s^{p-1}\mathrm ds.\]
Note that, since $\left((\xi_t,J_t),t \geq 0 \right)$ is a MAP, we can write for all $t\geq 0$ $I_t=e^{-\xi_t}I_{\xi'}$ where $(\xi',J')$ is, conditionally on $J_t$, a MAP with same distribution, with initial type $J_t$ and independent from $\xi_t$. Thus we can write
\[\E_i[I_1^p]=\sum_{j=1}^K \E_i[e^{-p\xi_1},J_t=j]\E_j[I_{\xi}^p]=\Big(e^{-\mathbf{\Phi}(p)}\mathbf{N}(p)\Big)_i\]
and similarly
\[\E_i[e^{-\xi_s}I_s^{p-1}]=\Big(e^{-s\mathbf{\Phi}(p)}\mathbf{N}(p-1)\Big)_i\]
We then end up with
\begin{align*}
\mathbf{N}(p)-e^{-\mathbf{\Phi}(p)}\mathbf{N}(p)&=p\Big(\int_{0}^1 e^{-s\mathbf{\Phi}(p)}\mathrm ds \Big)\mathbf{N}(p-1) \\
            &=p \big(\mathbf{\Phi}(p)\big)^{-1}\big(\mathbf{I}-e^{-\mathbf{\Phi}(p)}\big) \mathbf{N}(p-1).
\end{align*}
The use of the integration formula for the matrix exponential is justified by the fact that $\Phi(p)$ is invertible by Corollary \ref{inversible}. Similarly, note that by Proposition \ref{algebre}, the real parts of the eigenvalues of $-\mathbf{\Phi}(p)$ are strictly less than $\lambda(\mathbf{\Lambda})=0$, and thus the spectral radius of $e^{-\mathbf{\Phi}(p)}$ is strictly less than $1$, and $\mathbf{I}-e^{-\mathbf{\Phi}(p)}$ is invertible. Crossing it out, we end up with
\[\mathbf{N}(p)=p{\mathbf{\Phi}(p)}^{-1}\mathbf{N}(p-1).\]
\end{proof}

\noindent \textit{Proof of Proposition \ref{yu}.}
Point $(i)$ is proved by a straightforward induction, starting at $\mathbf{N}(0)=\mathbf{1}.$ $(ii)$ requires more work. Let $a>0$, we are interested in the nature of the matrix-valued series
\[\sum_{k=0}^{\infty}a^k \prod_{l=1}^k \big(\mathbf{\Phi}(k-l)\big)^{-1}.\]
For ease of notation, we let $\mathbf{A}_k=a\big(\mathbf{\Phi}(k)\big)^{-1}$ and $\mathbf{B}_k=\prod_{l=1}^k \mathbf{A}_{k-i},$ so that the series reduces to $\sum_{k=0}^{\infty} \mathbf{B}_k.$ By monotonicity, the matrix $\mathbf{\Phi}(k)$ converges as $k$ tends to infinity, and by monotonicity of its smallest eigenvalue (by Proposition \ref{algebre}), its limit is invertible. Thus $\mathbf{A}_k$ converges as $k$ tends to infinity to $\mathbf{M}=a\underset{k\to\infty}\lim (\mathbf{\Phi}(k))^{-1}$ and, for $a<\rho\big(\underset{k\to\infty}\lim (\mathbf{\Phi}(k))^{-1}\big),$ we have $\rho(\mathbf{M})<1$. Considering any subordinate norm $||\cdot||$ on the space of $K\times K$ matrices, we have by Gelfand's formula $\rho(\mathbf{M})=\underset{n\to\infty}\lim ||\mathbf{M}^n||^{1/n},$ and thus there exists $n$ such that $||\mathbf{M}^n||<1.$ By continuity of the product of matrices, we can find $\veps>0$ and $l_0\in\N$ such that
\[\forall l\geq l_0,\, ||\mathbf{A}_{l+n-1}\ldots\mathbf{A}_l||\leq 1-\veps.\]
Now, for $l\geq l_0+k-1$, let $\mathbf{C}_l=\mathbf{A}_l\ldots \mathbf{A}_{l-n+1},$ and notice that $||\mathbf{C}_l||\leq 1-\veps.$ For $k\in\N$ and $m\in\{0,\ldots,n-1\},$ write

\[\mathbf{B}_{l_0+kn+m}(\mathbf{B}_{l_0+m})^{-1}=\prod_{p=0}^{k-1} \mathbf{C}_{l_0+(k-p)n+m},\]
thus getting $||\mathbf{B}_{l_0+kn+m}(\mathbf{B}_{l_0+m})^{-1}||\leq (1-\veps)^n.$ Thus, for all $m\in\{0,\ldots,n-1\},$ the series
\[\sum_{k=0}^{\infty}\mathbf{B}_{l_0+kn+m}(\mathbf{B}_{l_0+m})^{-1} \]
converges absolutely, and hence the series
\[\sum_{k=0}^{\infty} \mathbf{B}_k\]
also converges.

\qed

\subsubsection{Negative moments}
In this section, we assume that there is no killing: $k_i=0$ for all $i$. We also assume $\underline{p}<0$, where $\underline{p}$ was defined in (\ref{defunderline}).

\begin{prop}\label{negativemoments} \quad
\begin{itemize}
\item[$(i)$] We have
\[\mathbf{N}(-1)=\mathbf{\Phi}'(0)\mathbf{1}+\mathbf{\Lambda N}'(0).\]
Where $\big(\mathbf{\Phi}'(0)\big)_{i,j}=\E_i[\xi_1,J_1=j]$ and $\big(\mathbf{N}'(0)\big)_i=\E_i[\ln I_{\xi}]$ for all $i,j.$

\item[$(ii)$] For an integer $k<0$ with $k>\underline{p}-1$, we then have
\[\mathbf{N}(k)=\frac{(-1)^{k+1}}{(|k|-1)!}\left(\prod_{l=k+1}^{-1}\mathbf{\Phi}(l)\right)\mathbf{N}(-1).\]
\end{itemize}
\end{prop}

As in the case of positive moments, the results come mostly from a recursion lemma.

\begin{lem}\label{negativemomentsrecursion} For $p\in (\underline{p},0)$, the entries of $\mathbf{N}(p-1)$ and $\mathbf{N}(p)$ are finite, and we have the recursion relation 
\[\mathbf{N}(p-1)=\frac{{\mathbf{\Phi}(p)}}{p}\mathbf{N}(p).\]
\end{lem}

\begin{proof} The proof of Lemma \ref{positivemomentsrecursion} does not apply directly and needs some modification. First, we check that the entries of $\mathbf{N}(p)$ are finite: for all $i$,
\[\E_i[I_{\xi}^p]\leq\E_i\Big[\Big(\int_0^1 e^{-\xi_t}\mathrm dt\Big)^{p}\Big]\leq \E_i[e^{p\xi_1}]<\infty.\]
The same steps as in the proof of Lemma \ref{positivemomentsrecursion} lead to
\[\big(\mathbf{I}-e^{-t\mathbf{\Phi}(p)}\big)\mathbf{N}(p)=p\Big(\int_{0}^t e^{-s\mathbf{\Phi}(p)}\mathrm ds\Big) \mathbf{N}(p-1)\]
for $t\geq 0$.
We deduce from this that the entries of $\mathbf{N}(p-1)$ are also finite: if at least one entry was infinite, then the right hand side would be infinite since $e^{-s\mathbf{\Phi}(p)}$ has positive entries for all $s>0$, and we already know that the left-hand side is finite.

We cannot compute the integral this time, so instead we take the derivative of both sides at $t=0$, and get
\[-{\mathbf{\Phi}(p)}\mathbf{N}(p)=-p\mathbf{N}(p-1),\]
thus ending the proof.
\end{proof}

\noindent \textit{Proof of Proposition \ref{negativemoments}.} Recalling that $\mathbf{\Phi}(0)=-\mathbf{\Lambda}$ (because of the lack of killing), $\mathbf{N}(0)=\mathbf{1},$ and $\mathbf{\Lambda}\mathbf{1}=0$, write
\[\mathbf{N}(p-1)=\frac{{\mathbf{\Phi}(p)}\mathbf{N}(p)-\mathbf{\Phi}(0)\mathbf{N}(0)}{p}.\]
Since $\mathbf{N}(p-1)$ is finite for at least some negative $p$, it is continuous when we let $p$ tend to $0$, and we end up with $\mathbf{N}(-1)=(\mathbf{\Phi N})'(0)=\mathbf{\Phi}'(0)\mathbf{N}(0)+\mathbf{\Phi}(0)\mathbf{N}'(0),$ which is what we need. Note that both $\mathbf{\Phi}$ and $\mathbf{N}$ are both differentiable at $0$, with derivatives being those mentioned in the statement of Proposition \ref{negativemoments}, because, respectively, $\xi_1$ has small exponential moments and $\E_i[I_{\xi}]$ and $\E_i[(I_{\xi})^{-1}]$ are both finite for all $i\in[K].$
\qed

\subsection{The Lamperti transformation and multi-type positive self-similar Markov processes}
\label{LampertiMAP}

In \cite{Lamp62, Lamp72}, Lamperti used a now well-known time-change to establish a one--to--one correspondence between L\'evy processes and non--negative self--similar Markov processes with a fixed index of self--similarity. It was generalised in \cite{CPR13} and \cite{ACGZ17} to real-valued and even $\R^d$-valued self-similar processes. We give here a variant adapted to our multi-type setting, which in fact coincides with the version presented in \cite{CPR13} when $K=2$. Let $\left((\xi_t,J_t),t \geq 0 \right)$ be a MAP and $\alpha\in\R$ be a number we call the \emph{index of self--similarity}. We let $\tau$ be the time--change defined by
	\[\tau(t) = \inf \left\{u, \int_0^u e^{\alpha\xi_r} \mathrm{d}r >t\right\},
\] 
and call \emph{Lamperti transform of $\left((\xi_t,J_t),t \geq 0 \right)$} the process $\left((X_t,L_t),t\geq 0)\right)$ defined by
\begin{equation}
\label{LMAP}
X_t=e^{-\xi_{\rho(t)}}, \quad L_t=J_{\rho(t)}.
\end{equation}
Note that, when $\alpha<0$, then $\tau(t)=\infty$ for $t\geq I_{|\alpha|\xi}$. In this case, we let by convention $X_{t}=0$ and  $L_t=0$. Note that, while $L$ is c\`adl\`ag on $[0,I_{|\alpha|\xi})$, it does not have a left limit at $I_{|\alpha|\xi})$ in general.

When $K=1$ and $\xi$ is a standard L\'evy process, $X$ is a non-negative self-similar Markov process, and reciprocally, any such Markov process can be written in this form, see  \cite{Lamp72}.
In general, for any $K$, one readily checks that the process $\left((X_t,L_t),t\geq 0)\right)$ is Markovian and $\alpha$-self-similar, in the sense that 
$\left((X_t,L_t),t\geq 0\right)$, started from $(x,i),$ has the same distribution as $ \left((xX'_{x^{-\alpha}t},J'_{x^{-\gamma}t}),t\geq 0\right)$, where  $\left((X'_t,L'_t),t\geq 0\right)$ is a version of the same process which starts at  $(1,i).$ This is justifies calling $\left((X_t,L_t),t\geq 0)\right)$ a \emph{multi-type positive self-similar Markov process} (mtpssMp). Since its distribution is completely characterised by $\alpha$ and the distribution of the underlying MAP, we will say that $\left((X_t,L_t),t\geq 0)\right)$ is the mtpssMp with characteristics $(\alpha,\mathbf{\Phi})$.

\section{Multi-type fragmentation processes}
Multi-type partitions and homogeneous multi-type fragmentations were introduced by Bertoin in \cite{Ber08}. We refer to this paper for more details on most of the definitions and results of Sections \ref{partitionsdef} and \ref{fragbasics}.
\subsection{Multi-type partitions}\label{partitionsdef}
We will be looking at two different kinds of partitions: mass partitions, which are simply partitions of the number $1$, and partitions of $\N$ and its subsets. In both cases, a type, that is an element of $\{1,\ldots,K\}$, is attributed to the blocks.

Let \[\mathcal{S}^{\downarrow}=\left\{\mathbf{s}=(s_n)_{n\in\N}: s_1\geq s_2\geq\ldots\geq0,\sum s_n\leq 1\right\}\]
be the set of nonnegative sequences which add up to at most $1$. This is the set of partitions used in the monotype setting, however here we will look at the set $\s$ which is formed of elements of the form $\bar{\mathbf{s}}=(s_n,i_n)_{n\in\N}\in \mathcal{S}^{\downarrow} \times \{0,1,\ldots,K\}^{\N}$ which are nonincreasing for the lexicographical ordering on $[0,1]\times \{0,1,\ldots,K\}$ and such that, for any $n\in\N$, $i_n=0$ if and only if $s_n=0$.

We interpret an element of $\s$ as the result of a particle of mass $1$ splitting into particles with respective sizes $(s_n,n\in\N)$ and types $(i_n,n\in\N).$ If $s_n=0$ for some $n$, we do not say that it corresponds to a particle with mass $0$ but instead that there is no $n$-th particle at all, and thus we give it a placeholder type $i_n=0$. We let $s_0=1-\sum_m s_m$ be the mass which has been lost in the splitting, and call it the \emph{dust} associated to $\bar{\mathbf{s}}$.

The set $\s$ is compactly metrised by letting, for two partitions $\bar{\mathbf{s}}$ and $\bar{\mathbf{s}}',$ $d(\bar{\mathbf{s}},\bar{\mathbf{s}}')$ be the Prokhorov distance between the two measures $s_0\delta_0+\sum_{n=1}^{\infty}s_{n}\delta_{s_n\mathbf{e}_{i_k}}$ and $s'_0\delta_0+\sum_{n=1}^{\infty}s'_{n}\delta_{s'_n\mathbf{e}_{i'_n}}$ on the $K$-dimensional unit cube (where $(\mathbf{e}_i,i\in[K])$ is the canonical basis of $\R^K$).

For $\bar{\mathbf{s}}\in\s$ and $p\in\R$ we introduce the row vector notation
\begin{equation}\label{crochet}
\bar{\mathbf{s}}^{\{p\}}=\sum_{n:s_n\neq 0}^{\infty} s_n^{p}\hspace{0.08cm}\mathbf{e}_{i_n}\in\R^{K}.
\end{equation}
Note that this is well-defined, since the set of summation is made to avoid negative powers of $0$.
\medskip

We call \emph{block} any subset of $\N$. For a block $B$, we let $\p_B$ be the set of elements of the type $\bar{\pi}=(\pi,\mathbf{i})=(\pi_n,i_n)_{n\in\N}$, where $\pi$ is a classical partition of $B$, its blocks $\pi_1,\pi_2,\ldots$ being listed in increasing order of their least element, and $i_n\in\{0,\ldots,K\}$ is the type of $n$-th block for all $n\in\N$, with $i_n=0$ if and only if $\pi_n$ is empty or a singleton.

A partition $\bar{\pi}$ of $B$ naturally induces an equivalence relation on $B$ which we call $\underset{\bar{\pi}}{\sim}$ by saying that, for two integers $n$ and $m$, $n\underset{\bar{\pi}}{\sim}m$ if an only if they are in the same block of $\pi$. The partition $\pi$ without the types can then be recovered from $\underset{\bar{\pi}}{\sim}.$

It will be useful at times to refer to the block of a partition containing a specific integer $n$. We call it $\pi_{(n)}$, and its type $i_{(n)}.$

If $A\subset B$, then a partition $\bar{\pi}$ of $B$ can be made into a partition of $A$ by restricting its blocks to $A$, and we call $\bar{\pi}\cap A$ the resulting partition. The blocks of $\bar{\pi}\cap A$ inherit the type of their parent in $\bar{\pi},$ unless they are empty or a singleton, in which case their type is $0$.

The space $\p_{\N}$ is classically metrised by letting, for two partitions $\bar{\pi}$ and $\bar{\pi}',$
\[d(\bar{\pi},\bar{\pi}')=\frac{1}{\sup\{n\in\N:\; \bar{\pi}\cap[n]=\bar{\pi}'\cap[n]\}}.\] This is an ultra-metric distance which makes $\p_{\N}$ compact.

A block $B$ is said to have an \emph{asymptotic frequency} if the limit
\[|B|=\lim_{n\to\infty} \frac{\#(B\cap [n])}{n}\]
exists. A partition $\bar{\pi}=(\pi,\mathbf{i})$ of $\N$ is then said to have asymptotic frequencies if all of its blocks have an asymptotic frequency. In this case we let $|\bar{\pi}|=(|\pi|,\mathbf{i})=(|\pi_n|,i_n)_{n\in\N}$ and $|\bar{\pi}|^{\downarrow}$ be the lexicographically decreasing rearrangement of $|\bar{\pi}|$, which is then an element of $\s.$

For any bijection $\sigma$ from $\N$ to itself and a partition $\bar{\pi}$, we let $\sigma\bar{\pi}$ be the partition whose blocks are the inverse images by $\sigma$ of the blocks of $\bar{\pi}$, each block of $\sigma\bar{\pi}$ inheriting the type of the corresponding block of $\bar{\pi}$. We say that a random partition $\overline{\Pi}$ is \emph{exchangeable} if, for any bijection $\sigma$ from $\N$ to itself, $\sigma\overline{\Pi}$ has the same distribution as $\overline{\Pi}$.

It was proved in \cite{Ber08} that Kingman's well-known theory for monotype exchangeable partitions (see \cite{Kingman}) has a natural extension to the multi-type setting. This theory summarily means that, for a mass partition $\bar{\mathbf{s}}=\mathbf{(s,i)}$, there exists an exchangeable random partition $\overline{\Pi}_{\bar{\mathbf{s}}},$ which is unique in distribution, such that $|\overline{\Pi}_{\bar{\mathbf{s}}}|^{\downarrow}=\bar{\mathbf{s}}$, and inversely, any exchangeable multi-type partition $\overline{\Pi}$ has asymptotic frequencies a.s., and, calling $\overline{\mathbf{S}}=|\overline{\Pi}|^{\downarrow},$ conditionally on $\overline{\mathbf{S}}$, the partition $\overline{\Pi}$ has distribution $\kappa_{\bar{\mathbf{S}}}.$

% Let $\bar{\mathbf{s}}=\mathbf{(s,i)}$ be a mass partition, and $(U_n,n\in\N)$ a sequence of independent uniform r.v. on $[0,1].$ We let $\overline{\Pi}_{\bar{\mathbf{s}}}$ be the unique partition such that two integers $n$ and $m$ are in the same block if there exists $k\in\N$ such that $U_n$ and $U_m$ are both in the interval $[\sum_{l=1}^{k-1}s_l, \sum_{l=1}^k s_l)$, and the type of this block is then $i_k$. We call the distribution of $\overline{\Pi}_{\bar{\mathbf{s}}}$ the paintbox distribution associated to $\bar{\mathbf{s}},$ and notate it $\kappa_{{\bar{\mathbf{s}}}}.$ Theorem 2.1 in \cite{Ber08} then states that any exchangeable multi-type partition $\overline{\Pi}$ has asymptotic frequencies a.s., and that, calling $\overline{\mathbf{S}}=|\overline{\Pi}|^{\downarrow},$ conditionally on $\overline{\mathbf{S}}$, the partition $\overline{\Pi}$ has distribution $\kappa_{\bar{\mathbf{S}}}.$

\subsection{Basics on multi-type fragmentations}\label{fragbasics}
\subsubsection{Definition}
Let $\overline{\Pi}=(\overline{\Pi}(t),t\geq0)$ be a c\`adl\`ag $\p_{\N}$-valued Markov process. We denote by $(\mathcal{F}^{\overline{\Pi}}_t,t\geq0)$ its canonical filtration, and, for $\bar{\pi}\in\p_{\N}$, call $\pr_{\bar{\pi}}$ the distribution of $\overline{\Pi}$ when its initial value is $\bar{\pi}.$ In the special case where $\bar{\pi}=(\N,i)$ has only one block, which has type $i\in[K]$, we let $\pr_i=\pr_{(\N,i)}.$ We also assume that, with probability $1$, for all $n\in\N$, $|\big(\Pi(t)\big)_{(n)}|$ exists for all $t\geq0$ and is a right-continuous function of $t$. Let also $\alpha\in\R$.
\begin{defn}
We say that $\overline{\Pi}$ is an $\alpha$-self-similar (or \emph{homogeneous} if $\alpha=0$) fragmentation process if $\overline{\Pi}$ is exchangeable as a process (i.e. for any permutation $\sigma$, the process $\sigma\overline{\Pi}=(\sigma\overline{\Pi}(t),t\geq0)$ has the same distribution has $\overline{\Pi}$) and satisfies the following $\alpha$-self-similar fragmentation property: for $\bar{\pi}=(\pi,\mathbf{i})\in\p_{\N},$ under $\pr_{\bar{\pi}},$ the processes $\Big(\overline{\Pi}(t)\cap\pi_n,t\geq 0\Big)$ for $n\in\N$ are all independent, and each one has the same distribution as $\Big(\overline{\Pi}(|\pi|_n^{\alpha}t)\cap\pi_n,t\geq 0\Big)$ has under $\pr_{i_n}.$ 
\end{defn}

We will for the sake of convenience exclude the degenerate case where the first component $(\Pi(t),t\geq0)$ is constant a.s, and only the type changes.

We will make a slight abuse of notation: for $n\in\N$ and $t\geq 0$, we will write $\Pi_n(t)$ for $(\Pi(t))_n$, and other similar simplifications, for clarity.

It will be convenient to view $\overline{\Pi}$ as a single random variable in the space $\mathcal{D}=\mathcal{D}(\,[0,+\infty),\p_{\N})$ of c\`adl\`ag functions from $[0,\infty)$ to $\p_{N}$, equipped with its usual Skorokhod topology. We also let, for $t\geq0$, $\mathcal{D}_t=\mathcal{D}([0,t],\p_{\N}),$ which will come of use later.

\smallskip

The Markov property can be extended to random times, even different times depending on which block we're looking at. For $n\in\N$, let $\mathcal{G}_n$ be the canonical filtration of the process $\big(|\Pi_{(n)}(t)|,i_{(n)}(t),t\geq0\big),$ and consider a $\mathcal{G}_n$-stopping time $L_n.$ We say that $L=(L_n,n\in\N)$ is a \emph{stopping line} if, moreover, for all $n$ and $m$, $m\in\Pi_{(n)}(L_n)$ implies $L_n=L_m,$ and use it to define a partition $\overline{\Pi}(L)$ which is such that, for all $n$, $(\overline{\Pi}(L))_{(n)}=(\overline{\Pi}(L_n))_{(n)}.$ We then have the following \emph{strong fragmentation property}: conditionally on $\big(\overline{\Pi}(L\wedge t),t\geq 0\big)$, the process $\big(\overline{\Pi}(L+t),t\geq0\big)$\footnote{It is straightforward to check that, if $L$ is a stopping line, then $(L_n\wedge t,n\in\N)$ and $(L_n\wedge t,n\in\N)$ also are stopping lines for all $t\geq 0$, justifying the definition of $\big(\overline{\Pi}(L\wedge t),t\geq 0\big)$ and $\big(\overline{\Pi}(L+t),t\geq0\big).$} has distribution $\pr_{\overline{\Pi}(L)}.$ We refer to \cite[Lemma 3.14]{BertoinBook} for a proof in the monotype case.

\subsubsection{Changing the index of self-similarity with Lamperti time changes}\label{sec:changement}
\begin{prop} Let $\overline{\Pi}$ be an $\alpha$-self-similar fragmentation process, and let $\beta\in\R.$ For $n\in\N$ and $t\geq 0$, we let
\[\tau_n^{(\beta)}(t) = \inf \left\{u, \int_0^u |\Pi_{(n)}(r)|^{-\beta} \mathrm{d}r >t\right\}.
\] 
For all $t\geq$, $\tau^{(\beta)}(t)=\big(\tau_n^{(\beta)}(t),n\in\N\big)$ is then a stopping line. Then, if we let
\begin{equation}\label{changement}
\overline{\Pi}^{(\beta)}(t)=\overline{\Pi}\big(\tau^{(\beta)}(t)\big),
\end{equation} $\overline{\Pi}^{(\beta)}$ is a self-similar fragmentation process with self-similarity index $\alpha+\beta.$
\end{prop}

For a proof of this proposition, we refer to the monotype case in \cite[Theorem 3.3]{BertoinBook}.

As a consequence, the distribution of $\overline{\Pi}$ is characterised by $\alpha$ and the distribution of the associated homogeneous fragmentation $\overline{\Pi}^{(-\alpha)}$.

\subsubsection{Poissonian construction}\label{poissoncons}
The work of Bertoin in \cite{Ber08} shows that a \emph{homogeneous} fragmentation has its distribution characterised by some parameters: a vector of non-negative \emph{erosion coefficients} $(c_i)_{i\in[K]}$, and a vector of \emph{dislocation measures} $(\nu_i)_{i\in[K]}$, which are sigma-finite measures on ${\s}$ such that, for all $i$,
	\[\int_{{\s}} (1-s_1\mathbbm{1}_{\{i_1=i\}})\mathrm d\nu_i(\bar{\mathbf{s}}) < \infty.
\]

Specifically, given a homogeneous fragmentation process $\overline{\Pi}$, there exists a unique set of parameters $\big(c_i,\nu_i,i\in[K]\big)$ such that, for any type $i$, the following construction gives a version of $\overline{\Pi}$ under $\pr_i$. For all $j\in[K]$, let $\kappa_{\nu_j}=\int_{\s}\kappa_{\bar{\mathbf{s}}}\mathrm d\nu_j(\bar{\mathbf{s}})$ (recalling that $\kappa_{\mathbf{\bar{s}}}$ is the paintbox measure on ${\p}_{\N}$ associated to $\mathbf{\bar{s}}$), and, for $n\in\N$, we let $(\overline{\Delta}^{(n,j)}(t),t\geq0)=\big((\Delta^{(n,j)}(t),\delta^{(n,j)}(t)),t\geq0\big)$ be a Poisson point process with intensity $\kappa_{\nu_j}$, which we all take independent. Recall that this notation means that $\delta^{(n,j)}_m(t)$ is the type given to the $m$-th block of the un-typed partition $\Delta^{(n,j)}(t).$ Now build $\overline{\Pi}$ under $\pr_i$ thus:
\begin{itemize}
\item Start with $\overline{\Pi}(0)=\mathbf{1}_{\N,i}.$
\item For $t\geq0$ such that there is an atom $\overline{\Delta}^{(n,j)}(t)$ with $i_n(t^-)=j$, replace $\overline{\Pi}_n(t^-)$ by its intersection with $\overline{\Delta}^{(n,j)}(t)$.
\item Send each integer $n$ into a singleton at rate $c_{i_{(n)}(t)}$.
\end{itemize}
This process might not seem well-defined, since the set of jump times can have accumulation points. However the construction is made rigorous in \cite{Ber08} by noting that, for $n$ in $\N$, the set of jump times which split the block $\overline{\Pi}\cap[n]$ is discrete, thus $\overline{\Pi}(t)\cap[n]$ is well-defined for all $t\geq0$ and $n\in\N$, and thus $\overline{\Pi}(t)$ also is well-defined for all $t\geq0$.

\medskip

As a consequence, the distribution of any self-similar fragmentation process $\overline{\Pi}$ is characterised by its index of self-similarity $\alpha,$ the erosion coefficients  $(c_i)_{i\in[K]}$ and dislocation measures $(\nu_i)_{i\in[K]}$ of the homogeneous fragmentation $\overline{\Pi}^{(-\alpha)}$. This justifies saying from now on that $\overline{\Pi}$ is a self-similar fragmentation with characteristics $\big(\alpha,(c_i)_{i\in[K]},(\nu_i)_{i\in[K]}\big).$

\subsubsection{The tagged fragment process} For $t\geq 0$, we call \emph{tagged fragment} of $\overline{\Pi}(t)$ its block containing $1$. We are interested in its size and type as $t$ varies, i.e. the process $\big((|\Pi_1(t)|,i_1(t)),t\geq0\big).$ It is in fact a mtpssMp, with characteristics $(\alpha,\mathbf{\Phi})$, where $\mathbf{\Phi}$ is given by
\begin{equation}\label{tagmatrix}
\mathbf{\Phi}(p)=\big(c_i(p+1)\big)_{\mathrm{diag}}+\left(\int_{\s}\left(\mathbbm{1}_{\{i=j\}}-\sum_{n=1}^{\infty}s_n^{1+p}\mathbbm{1}_{\{i_n=j\}}\right)\nu_i(\mathrm d \bar{\mathbf{s}})\right)_{i,j\in[K]}.
\end{equation}

This is proven in \cite{Ber08} when $\alpha=0$ and $c_i=0$ for all $i$ by using the Poissonian construction, however, after taking into account the Lamperti time-change, the proof does not differ significantly in the general case.

One consequence of exchangeability is that, for any $t\geq0$, conditionally on the mass partition $|\overline{\Pi}(t)|^{\downarrow},$ the tagged fragment is a size-biased pick amongst all the fragments of $|\overline{\Pi}(t)|^{\downarrow}.$ We thus have, for any non-negative measurable function $f$ on $[0,1]$ and $j\in[K],$
\begin{equation}\label{tag}
\E_i\left[f\big(|\Pi_1(t)|\big),i_1(t)=j\right]=\sum_{n\in\N}\E_i\left[ |\Pi_n(t)|f\Big(|\Pi_n(t)|\Big),i_n(t)=j\right].
\end{equation}
(recall that the blocks in the right-hand side of (\ref{tag}) are ordered in increasing order of their smallest element.)

We end this section with a definition: we say that the fragmentation process $\overline{\Pi}$ is \emph{irreducible} if the Markov chain of types in MAP associated to the tagged fragment is irreducible in the usual sense.

\subsection{Malthusian hypotheses and additive martingales}\label{sec:malt}
In this section and the next, we focus on the homogeneous case: we fix $\alpha=0$ until Section \ref{sec:extinction}. Recall, for $\bar{\mathbf{s}}\in\s$ and $p\in\R$, the notation $\bar{\mathbf{s}}^{\{p\}}$ from (\ref{crochet}).
\begin{prop} For all $p>\underline{p}+1$, the row matrix process $\big(\mathbf{M}(t),t\geq 0\big)$ defined by
\[ \mathbf{M}(t)=|\overline{\Pi}(t)|^{\{p\}}e^{t\mathbf{\Phi}(p-1)}\]
is a martingale.
\end{prop}

\begin{proof}
Let $t\geq0$ and $s\geq 0$, and $i,j$ be two types. Calling $\overline{\Pi}'$ an independent version of $\overline{\Pi},$ we have, by the fragmentation property at time $t$, and then exchangeability,
\begin{align*}
\E_i\left[\sum_n |\Pi_n(t+s)|^p\mathbbm{1}_{\{i_n(t+s)=j\}}\mid \mathcal{F}_t\right]
         &= \sum_n |\Pi_n(t)|^p \,\E_{i_n(t)}\left[\sum_m |\Pi'_m(s)|^p\mathbbm{1}_{\{i'_m(s)=j\}}\right] \\
         &= \sum_n |\Pi_n(t)|^p \,\E_{i_n(t)}\left[|\Pi'_1(s)|^{p-1}\mathbbm{1}_{\{i'_1(s)=j\}}\right] \\
         &=\sum_n |\Pi_n(t)|^p \left(e^{-s\mathbf{\Phi}(p-1)}\right)_{i_n(t),j}.
\end{align*}
Hence\[\E_i \left[|\overline{\Pi}(t+s)|^{\{p\}}\mid \mathcal{F}_t\right]=|\overline{\Pi}(t)|^{\{p\}}e^{-s\mathbf{\Phi}(p-1)},\] and thus $\mathbf{M}(t)$ is a martingale.
\end{proof}

\begin{cor}\label{1dimmart} Assume that the fragmentation is irreducible. We can then let 
\[\lambda(p)=-\lambda(-\mathbf{\Phi}(p-1)),\] where we use in the notation of Proposition \ref{algebre} in the right-hand side (i.e $\lambda(p)$ is the smallest eigenvalue of $\mathbf{\Phi}(p-1))$). Let $\mathbf{b}(p)=(b_i(p))_{i\in [K]}$ be a corresponding positive eigenvector (which is unique up to constants). Then, for $i\in[K]$, under $\pr_i$, the process $\big(M(t),t\geq 0\big)$ defined by
\[M(t)=\frac{1}{b_i(p)}e^{t\lambda(p)}\mathbf{M}(t)\mathbf{b}(p)=\frac{1}{b_i(p)}e^{t\lambda(p)}\sum_{i=1}^K \big(\mathbf{M}(t)\big)_ib_i(p)=\frac{1}{b_i(p)}e^{t\lambda(p)}\sum_{n=1}^{\infty} |\Pi_n(t)|^pb_{i_n(t)}(p) \]
is also a martingale, which we call the \emph{additive martingale} associated to $p$.
\end{cor}

\begin{defn} We say that the fragmentation process (or the characteristics $\big((c_i)_{i\in[K]},(\nu_i)_{i\in[K]}\big))$ is \emph{Malthusian} if it is irreducible and there exists a number $p^*\in(0,1]$ called the \emph{Malthusian exponent} such that
\[\lambda(p^*)=0.\]
\end{defn}

\begin{rem}
$(i)$ This definition, while fairly complex, is indeed the approriate generalisation of the Malthusian hypothesis for monotype fragmentations (see for example \cite{BertoinBook}). In particular, typical Malthusian cases are those where $c_i=0$ for all $i$ and the measures $(\nu_i)$ are all conservative, that is $\nu_i\big(\{s_0>0\}\big)=0$ for all $i.$ In this case, the MAP underlying the tagged fragment process is not killed, and thus $p^*=1$ by Corollary \ref{inversible}.

$(ii)$ Note that $\lambda$ is strictly increasing and continuous on $(\underline{p}+1,1]$. In particular, $p^*$ must be unique.
\end{rem}

Here are two examples of Malthusian cases.

\begin{ex}\label{example0} Assume that there exists $q\in(0,1]$ such that, for all $i\in[K]$,
\[c_iq + \int_{\s}\big(1-\sum_{n=1}^{\infty} s_i^q \big)\mathrm d \nu_i(\bar{\mathbf{s}})=0.\]
Then the characteristics $\big((c_i)_{i\in[K]},(\nu_i)_{i\in[K]}\big))$ are Malthusian, with Malthusian exponent equal to $q$.
\end{ex}
Example \ref{example0} says that, if, when we forget the types of the children of a particle, the corresponding  monotype Malthusian exponent is informally $q$ independently of the type of the parent, then the multi-type fragmentation process also has Malthusian exponent $q$.

\begin{ex}\label{example} Assume for all $j\in[K]$ that $c_j=0$ and $\nu_j$ has total mass $1$, and is fully supported by
\[\left\{\bar{\mathbf{s}}\in: \; \forall n, i_n=0 \text{ or } j+1,\: \text{ and } \sum_{n=1}^N s_n=1\right\}.\]
($j+1$ is taken modulo $K$, the the sense that $K+1=1$.) In words, each splitting preserves total mass, only has at most $N$ blocks, and the types evolve in a cyclic fashion.

For each $j\in[K]$, assume that $\nu_j$ is Malthusian ``if we forget the types", in the sense that there exists $p^*_j\in[0,1]$ such that
\[\int_{\s} \big(1-\sum_{n=1}^{\infty}s_n^{p^*}\big)\mathrm d\nu_j(\bar{\mathbf{s}})=0.\]

The multi-type fragmentation process with characteristics $\big((0)_{i\in[K]},(\nu_i)_{i\in[K]}\big))$ is then also Malthusian, and its Malthusian exponent $p^*$ satisfies $\min p^*_j\leq p^*\leq \max p^*_j.$

\end{ex}
Note that our assumptions do not exclude that, for some (but not all) $j\in[K]$, $\nu_j=\delta_{1,j+1}$, in which case we let $p_j^*=0$.

We postpone the proofs of these examples to Appendix \ref{Appendix}.

\medskip

We will now restrict ourselves to $p=p^*,$ and let $b_j=b_j(p^*)$ for all $j\in [K]$. In particular, the additive martingale can be rewritten as 
\begin{equation}\label{additivemartingale}
M(t)=\frac{1}{b_i}\sum_{n=1}^{\infty} |\Pi_n(t)|^{p^*}b_{i_n(t)}.
\end{equation}

This non-negative martingale has an a.s. limit $W=\underset{t\to\infty}\lim M(t)$. This convergence however is not strong enough for our purposes here, so, for $q>1$, we introduce the stronger Malthusian assumption $(\mathbf{M}_q),$ that for all $i\in[K],$
\begin{equation}\tag{$\mathbf{M}_q$}\label{Mq}	
\int_{\s}\left|1-\sum_{n=1}^{\infty}s_n^{p^*}\right|^q d\nu_i(\bar{\mathbf{s}}) <\infty.
\end{equation}

\begin{prop}\label{Lq} Assume \eqref{Mq} for some $q>1$. Then the martingale $\big(M(t),t\geq 0\big)$ converges to $W$ in $L^q$.
\end{prop}
\begin{proof} By the same arguments as in \cite[Proposition 4.4]{Steph13}, we only need to show that the sum of the $q$-th powers of the jumps of $\big(M(t),t\geq 0\big)$ has finite expectation:
\[\E_i\Big[\sum_{t\geq 0} |M(t)-M(t^-)|^q\Big]<\infty.\]

We compute this expectation with the Master formula for Poisson point processes (see \cite{RY}, page 475). Recalling the Poissonian construction of the fragmentation process in Section \ref{poissoncons}, we can write

\begin{align*}
\E_i\Big[\sum_{t\geq 0} |M(t)-M(t^-)|^q\Big]&=\E_i\left[\sum_{n=1}^{\infty}\sum_{t\geq0}|\Pi_n(t^-)|^{qp^*}\left(|1-\sum_{m=1}^{\infty}|\Delta^{n,i_n(t^-)}_m(t)|^{p^*}|\right)^q\right] \\
    &= \E_i\left[\int_0^{\infty}\sum_{n=1}^{\infty}|\Pi_n(t^-)|^{qp^*}\int_{\s}|1-\sum_{m=1}^{\infty} s_m^{p^*}|^q \mathrm d\nu_{i_n(t^-)}(\mathbf{\bar{s}})\mathrm d t\right]  \\
    &\leq \E_i\left[\int_0^{\infty}\sum_{n=1}^{\infty}|\Pi_n(t^-)|^{qp^*}\mathrm dt\right] \underset{j\in[K]}\sup \int_{\s}|1-\sum_{m=1}^{\infty} s_m^{p^*}|^q \mathrm d\nu_j(\mathbf{\bar{s}}) \\
    &= \E_i\left[\int_0^{\infty}\sum_{n=1}^{\infty}|\Pi_n(t)|^{qp^*}\mathrm dt\right] \underset{j\in[K]}\sup \int_{\s}|1-\sum_{m=1}^{\infty} s_m^{p^*}|^q \mathrm d\nu_j(\mathbf{\bar{s}}).
\end{align*}

Recall that, by Corollary \ref{1dimmart} applied to $p=qp^*$, we have, for all $t\geq 0,$ $\E_i[\sum_{n=1}^{\infty}b_{i_n(t)}(qp^*)|\Pi_n(t)|^{qp^*}]=b_i(qp^*)e^{-t\lambda(qp^*)},$ and so there exists a constant $C>0$ (depending on $q$) such that, for $t\geq 0$, 
\[\E_i[\sum_{n=1}^{\infty}|\Pi_n(t)|^{qp^*}]\leq Ce^{-t\lambda(qp^*)}.\] Since $q>1$, we have  $\lambda(qp^*)>0$ by monotonicity of $\lambda$, hence by Fubini's theorem
\[ \E_i\left[\int_0^{\infty}\sum_{n=1}^{\infty}|\Pi_n(t)|^{qp^*}\mathrm dt\right]\leq C\int_0^{\infty} e^{-t\lambda(qp^*)}\mathrm dt <\infty,\]
ending the proof.
\end{proof}

\begin{lem}\label{extinction} Assume that the additive martingale converges to $W$ in $L^1.$ Then, a.s., if $W\neq 0$, then $\overline{\Pi}$ does not get completely reduced to dust in finite time.
\end{lem}

\begin{proof}
This kind of result is well-known, but not in multi-type settings, so we will give the details. For $n\in\Z_+$, and $j\in[K]$, let $Z^{(j)}(n)$ be the number of blocks of $\overline{\Pi}(n)$ with type $j$. Calling $\mathbf{Z}(n)=(Z^{(j)}(n),j\in[K]),$ the process $(\mathbf{Z}(n),n\in\N)$ is then a multi-type Galton-Watson process, see \cite[Chapter II]{Ha} for an introduction. By irreducibility of $\Pi$, $(\mathbf{Z}(n),n\in\N)$ is positive in the sense that $\pr_i[Z^{(j)}(1)>0]$ is positive for all $i,j\in[K]$. Assume that it is supercritical (otherwise $W=0$ a.s. and there is nothing to do). Let, for $i\in[K]$, $f^{(i)}$ be the generating function defined by $f^{(i)}(\mathbf{x})=\E_i\big[\prod_{j=1}^K x_j^{Z^{(j)}(1)}\big]$ for $\mathbf{x}=(x_1,\ldots,x_k)\in (\R_+)^K,$ and $p_i=\pr_i[W=0]$, and group these in $\mathbf{f}=(f^{(i)},i\in[K])$ and $\mathbf{p}=(p_i,i\in[K]).$ One then readily has
\[\mathbf{p}=\mathbf{f}(\mathbf{p}),\]
which implies by \cite[Corollary 1 of Theorem 7.2]{Ha} that $\pr_i[W=0]$ is either equal to $1$ or equal to the probability of extinction starting from type $i$. But since $\E_i[W]>0$ by $L^1$-convergence, $\pr_i[W=0]\neq 1$, and thus $W\neq 0$ a.s. on nonextinction of $(\mathbf{Z}_n,n\in\N).$
\end{proof}

\subsection{Biasing}\label{biasing}
For $t\geq0$, we let $\pr^*_{i,t}$ be the probability measure on $\mathcal{D}_t=\mathcal{D}([0,t],\p_{\N})$ with corresponding expectation operator $\E^*_{i,t}$ be defined by
\[\E^*_{i,t}\left[F(\overline{\Pi}(s),0\leq s\leq t)\right]=\frac{1}{b_i}\E_i\left[b_{i_1(t)}|\Pi_1(t)|^{p^*-1}F\big(\overline{\Pi}(s),0\leq s\leq t\big)\right]\]
for a nonnegative measurable function $F$ on $\mathcal{D}_t.$
One classically checks that, because of the martingale property of $b_{i_1(t)}|\Pi_1(t)|^{p^*-1},$ these measures are compatible, and by Kolmogorov's extension theorem, there exists a unique probability measure $\pr_i^*$ on $\mathcal{D}$ such that, for all $t\geq 0$ and $F$ a nonnegative measurable function on $\mathcal{D}_t$,
\[\E^*_i\left[F(\overline{\Pi}(s),0\leq s\leq t)\right]=\E^*_{i,t}\left[F(\overline{\Pi}(s),0\leq s\leq t)\right].\]

Let us give another way of interpreting $\pr_{i,t}$. For $n\in\N$ and $s\leq t,$ let $\overline{\Psi}^n(s)$ be the same partition as $\overline{\Pi}(s)$, except that, for $n\geq 2$, the integer $1$ has changed blocks: it is put in $\Pi_n(t).$ We then define a new measure $\pr_{i,t}^{\bullet}$ by

\[\E_{i,t}^{\bullet}\left[F\big(\overline{\Pi}(s),0\leq s\leq t)\right]=\frac{1}{b_i}\E\left[\sum_{n\in\N} b_{i_n}|\Pi_n(t)|^{p^*}F\big(\overline{\Psi}^n(s),0\leq s\leq t\big)\right]. \]

\begin{prop}
The two distributions $\pr_{i,t}^*$ and $\pr_{i,t}^{\bullet}$ are equal.
\end{prop}
The proof is elementary but fairly heavy, so we refer the reader to \cite{Steph13} for the monotype case, which is easily generalised.

\medskip

As with $\pr_i$, there is a way of using Poisson point processes to construct the measure $\pr^*_i.$ The method is the same as in Section \ref{poissoncons}, with one difference: for all $j\in[K]$, the point process $(\Delta^{(1,j)}(t),t\geq0)$ has intensity $\kappa_{\nu_j}^*$ instead of $\kappa_{\nu_j}$, where the measure $\kappa_{\nu_j}^*$ is defined by
	\[\mathrm{d}\kappa_{\nu_j}^*(\bar{\pi})=\frac{1}{b_j}b_{i_1}|\pi_1|^{p^*-1}\mathbf{1}_{\{|\pi_1|\neq0\}}\mathrm{d}\kappa_{\nu_j}(\bar{\pi}).
\]
The construction is still well defined, because, for any $k\in\N$,
\begin{align*}
\kappa_{\nu_j}^*(\{[k] \text{ is split into two or more blocks}\})&=\frac{1}{b_j}\int_{\s}  (1-\sum_{n=1}^\infty s_n^k)\sum_{n=1}^\infty b_{i_n}s_n^{p^*}\mathrm  d\nu_j(\bar{\mathbf{s}}) \\
                            &=\frac{1}{b_j}\left(c_jp^*+ \int_{\s} (1-\sum_{n=1}^\infty s_n^k\sum_{n=1}^\infty b_{i_n}s_n^{p^*})\mathrm d\nu_j(\bar{\mathbf{s}})\right) \\
                            &\leq \frac{1}{b_j}\left(c_jp^*+\int_{\s} (1-m s_1^{p^*+k})\mathrm  d\nu_j(\bar{\mathbf{s}})\right) \\
                            &<\infty,
\end{align*}
where $m=\underset{i\in[K]}\min b_i$ is positive.

We omit the proof that this modified Poisson construction does produce the distribution $\pr_i^*$. The reader can check the proof of Theorem 5.1 in \cite{Steph13} for the monotype case.

\smallskip

This biasing procedure also changes the distribution of the tagged fragment process. It is still a MAP, but has a modified Bernstein matrix.

\begin{prop} Under $\pr^*_i$, the process $\Big((-\log|\Pi_1(t)|,i_1(t)),t\geq0\Big)$ is a MAP with Bernstein matrix $\mathbf{\Phi}^*$, defined by 
\begin{equation}\label{phistar}
\Big(\mathbf{\Phi}^*(p)\Big)=\Big((b_i)_\mathrm{diag}\Big)^{-1}\Big(\mathbf{\Phi}(p+p^*-1)\Big)\Big((b_i)_\mathrm{diag}\Big)
\end{equation} for $p\geq 0$.
\end{prop}

\begin{proof} That we have the correct moments is straightforward to check. Let $p\geq 0$ and $j\in[K]$, we have by definition
\[\E_i^*\big[|\Pi_1(t)|^{p},i_1(t)=j\big]=\frac{b_j}{b_i}\E_i\big[|\Pi_1(t)|^{p^*-1}|\Pi_1(t)|^{p},i_1(t)=j\big]=\Big(\mathbf{\Phi}^*(p)\Big)_{i,j}.\]
The same definition is also enough to prove that $\Big((-\log|\Pi_1(t)|,i_1(t)),t\geq0\Big)$ is indeed a MAP. Let $s<t$ and let $F$ be a function on $\mathcal{D}$ taking the form $F(\bar{\pi})=f\Big(\frac{|\pi_1(t)|}{|\pi_1(s)|}\Big)G\Big(\bar{\pi}(r),0\leq r \leq s\Big)\mathbbm{1}_{\{i_1(t)=j\}},$ and write

\begin{align*}
&\E_i^*\left[f\Big(\frac{|\Pi_1(t)|}{|\Pi_1(s)|}\Big) G\big(\overline{\Pi}(r),0\leq r \leq s\big), i_1(t)=j \right]  \\
   &\qquad\qquad=\frac{b_j}{b_i}\E_i\left[|\Pi_1(t)|^{p^*}f\Big(\frac{|\Pi_1(t)|}{|\Pi_1(s)|}\Big)G(\overline{\Pi}(r),0\leq r \leq s),i_1(t)=j\right] \\
   &\qquad\qquad=\frac{b_j}{b_i}\E_i\left[\Big(\frac{|\Pi_1(t)|}{|\Pi_1(s)|}\Big)^{p^*}f\Big(\frac{|\Pi_1(t)|}{|\Pi_1(s)|}\Big)|\Pi_1(s)|^{p^*}G(\overline{\Pi}(r),0\leq r \leq s),i_1(t)=j\right] \\
   &\qquad\qquad=\frac{b_j}{b_i}\E_i\left[|\Pi_1(s)|^{p^*}G(\overline{\Pi}(r),0\leq r \leq s)\E_{i_1(s)}\Big[|\Pi'_1(t-s)|^{p^*}f(|\Pi'_1(t-s)|),i'_1(t-s)=j\Big]\right] \\
   &\qquad\qquad=\E_i\left[\frac{b_{i_1(s)}}{b_i}|\Pi_1(s)|^{p^*}G(\overline{\Pi}(r),0\leq r \leq s)\E_{i_1(s)}\Big[\frac{b_j}{b_{i_1(s)}}|\Pi'_1(t-s)|^{p^*}f(|\Pi'_1(t-s)|),i'_1(t-s)=j\Big]\right] \\
   &\qquad\qquad=\E_i^*\left[G(\overline{\Pi}(r),0\leq r \leq s)\E_{i_1(s)}^*\left[f(|\Pi'_1(t-s)|),i'_1(t-s)=j\right]\right].
\end{align*}
Note that the third equality comes the fact that $\Big((-\log|\Pi_1(t)|,i_1(t)),t\geq0\Big)$ is a MAP under $\pr_i$, while the last one is what we are looking for: it shows that $\Big((-\log|\Pi_1(t)|,i_1(t)),t\geq0\Big)$ is a MAP under $\pr_i^*.$
\end{proof}

\begin{rem} This can be seen as a \emph{spine decomposition} of the fragmentation process: the fragment containing $1$ is the spine, and dislocates with a special biased rate, and all the other fragments evolve with the usual branching mechanism.
\end{rem}

\subsection{Extinction when the index of self-similarity is negative}\label{sec:extinction}
In this section, $\overline{\Pi}$ is an $\alpha$-self-similar fragmentation with $\alpha<0.$ In this case, we already know from Section \ref{LampertiMAP} that the size of the tagged fragment will reach $0$ in finite time. However, a much stronger result is true:
\begin{prop}\label{deathtimeexpmoments} Let $\zeta=\inf\Big\{t\geq0: \Pi(t)=\big\{\{1\},\{2\},\ldots\big\}\Big\}.$ Then $\zeta$ is finite a.s. and has some finite exponential moments.
\end{prop}
\begin{proof} We follow the idea of the proof of Proposition 14 in \cite{H03}, our main tool being the fact that the death time of the tagged fragment in a self-similar fragmentation with index of similarity $\alpha/2$ also has exponential moments by Proposition \ref{yu}, since it is the exponential functional of a MAP.

Fix a starting type $i\in[K]$. For $t\geq 0$, let
\[X(t)=\underset{n\in\N}\max\, |\Pi^{(-\alpha/2)}_n(t)|\]
be the largest asymptotic frequency of a block of $\overline{\Pi}^{(-\alpha/2)}(t),$ where $\overline{\Pi}^{(-\alpha/2)}$ is the $\alpha/2$-self-similar fragmentation obtained by Section \ref{sec:changement} with $\beta=-\alpha/2$. Doing the time-change which transforms $\overline{\Pi}^{(-\alpha/2)}$ into $\overline{\Pi},$ we obtain
\[\zeta\leq \int_0^{\infty} X(r)^{-\alpha/2}\mathrm d r.\]

We then can write, for $t\geq 0$,
\begin{align*}
\pr_i[\zeta>2t]&\leq \pr_i\Big[\int_0^{\infty}X(r)^{-\alpha/2}\mathrm d r>2t\Big] \\
             &= \pr_i\Big[\int_0^{\infty}X(rt)^{-\alpha/2}\mathrm d r>2\Big] \\
             &\leq \pr_i\Big[\int_1^{\infty}X(rt)^{-\alpha/2}\mathrm d r>1\Big] \\
             &\leq \int_1^{\infty} \E_i[X(rt)^{-\alpha/2}] \mathrm d r.
\end{align*}

If $\alpha\leq -2$, then using (\ref{tag}), we get
\[
\E_i[X(rt)^{-\alpha/2}]\leq \E_i[X(rt)] \leq \E_i\Big[\sum_{n\in\N} |\Pi^{(-\alpha/2)}_n|(t)\Big]=\pr_i\Big[|\Pi_1^{(-\alpha/2)}(t)|\neq 0\Big]
\]
where $|\Pi_1(t)^{(-\alpha/2)}|$ is the mass of tagged fragment of $\overline{\Pi}^{(-\alpha/2)}$ at time $t$.

If $\alpha>-2$, then by Jensen's inequality, and (\ref{tag}) again,
\[\E_i[X(rt)^{-\alpha/2}]\leq \big(\E_i[X(rt)]\big)^{-\alpha/2}\leq \left(\E_i\Big[\sum_{n\in\N} |\Pi^{(-\alpha/2)}_n|(rt)\Big]\right)^{-\alpha/2}=\left(\pr_i\Big[|\Pi_1^{(-\alpha/2)}(rt)|\neq 0\Big]\right)^{-\alpha/2}. \]
Since $\overline{\Pi}^{(-\alpha/2)}$ is a self-similar fragmentation with negative index $\alpha/2$, the death time of $|\Pi_1(t)^{(-\alpha/2)}|$ has exponential moments by Proposition \ref{yu}. As a consequence, both for $\alpha \leq -2$ and $\alpha>-2$, there exists constants $A$ and $B$ such that, for all $t \geq 0,$

\[\E_i[X(rt)^{-\alpha/2}]\leq A e^{-Brt}.\]

Integrating with respect to $r$ from $1$ to infinity then yields
\begin{align*}
\pr_i[\zeta>2t]&\leq A\int_{1}^{\infty} e^{-Brt}\mathrm dr \\
&\leq \frac{A}{Bt} e^{-Bt},
\end{align*}
which is enough to conclude.

\end{proof}

\section{Multi-type fragmentation trees}
In this section, we will go back an forth between homogeneous and self-similar fragmentations, so we use adapted notations: $\overline{\Pi}$ will be a homogeneous fragmentation process, and $\overline{\Pi}^{(\alpha)}$ will be the $\alpha$-self-similar process obtained using Section \ref{sec:changement}.

\subsection{Vocabulary and notation concerning $\R$-trees}
\noindent\textbf{Basic definitions}
\begin{defn} Let $(\T,d)$ be a metric space. We say that it is an \emph{$\R$-tree} if it satisfies the following two conditions:

\textbullet \hspace{2 mm} For all $x,y \in \T$, there exists a unique distance-preserving map $\phi_{x,y}$ from $[0,d(x,y)]$ into $\mathcal{T}$ such  $\phi_{x,y}(0)=x$ and $\phi_{x,y}(d(x,y))=y.$

\textbullet \hspace{2 mm} For all continuous and one-to-one functions $c$: $[0,1] \to \T $, we have $\\ c([0,1]) = \phi_{x,y}([0,d(x,y)]),$ where $x=c(0)$ and $y=c(1)$.

\end{defn}
For any $x$ and $y$ in $\T$, we will denote by $\llbracket x,y\rrbracket$ the image of $\phi_{x,y}$, i.e. the path between $x$ and $y$.

We usually consider trees which are \emph{rooted} and measured, that is which have a distinguished vertex $\rho$ called the \emph{root}, and are equipped with a Borel probability measure $\mu$. The root being fixed, this lets us define a \emph{height function} on $\T$ as $ht(x)=d(\rho,x)$ for $x\in\T$.

A \emph{leaf} of $\T$ is any point $x$ different from the root, such that $\T\setminus\{x\}$ is connected.

When there is no ambiguity, we usually drop the metric, root and measure from the notation, just writing $\T$ for $(\T,d,\rho,\mu)$. For $a>0$, we let $a\T$ be the rescaled $\R$-tree $(\T,ad).$

We introduce some more notation to easily refer to some subsets and points of $\T$: for $x\in\T$, we let $\T_x=\{y \in\T: x\in\llbracket \rho,y\rrbracket\}$ be the subtree of $\T$ rooted at $x$. %For $t\geq 0$, we also let $\T_{\leq t}=\{x\in\T: ht(x)\leq t\},$ as well as the similarly defined $\T_{<t}$, $\T_{\geq t}$, $\T_{>t}.$ 
If $y\in\T$, we also let $x\wedge y$ be the infimum of $x$ and $y$ for the natural order on $\T$, i.e. the point at which the paths $\llbracket\rho,x\rrbracket$ and $\llbracket\rho,y\rrbracket$ separate from one another.

\medskip

\noindent\textbf{Gromov-Hausdorff-Prokhorov topology.} Two compact rooted and measured $\R$-trees $(\T,d,\rho,\mu)$ and $(\T',d',\rho',\mu')$ can be compared using the well-known \emph{Gromov-Hausdorff-Prokhorov metric} $d_{GHP}$ defined by
  \[d_{GHP}(\T,\T') = \inf [ \max (d_{\mathcal{Z},H} (\phi(\T),\phi'(\T')), d_\mathcal{Z}(\phi(\rho),\phi'(\rho')),d_{\mathcal{Z},P}(\phi_*\mu,\phi'_*\mu')],
\]
where the infimum is taken over all pairs of isometric embeddings $\phi$ and $\phi'$ of $\T$ and $\T'$ in the same metric space $(\mathcal{Z},d_{\mathcal{Z}})$, $d_{\mathcal{Z},H}$ is the Hausdorff distance between closed subsets of $\mathcal{Z}$, $d_{\mathcal{Z},P}$ is the Prokhorov distance between Borel probability measures on $Z$, and $\phi_*\mu$ and $\phi'_*\mu'$ are the respective image measures of $\mu$ and $\mu'$ by $\phi$ and $\phi'$.

It is well-known that $d_{GHP}$ makes the space $\TT_W$ of equivalence classes of compact, rooted and measured trees (up to metric isomorphisms which preserve the roots and measures) a compact metric space, see \cite{EPW06} and \cite{ADH}.

\medskip

\noindent\textbf{Defining a measure on an $\R$-tree using nonincreasing functions.} In \cite{Steph13} was given a useful tool to define a Borel measure on a compact rooted tree $\T$. Let $m$ be a nonincreasing function from $\T$ to $[0,\infty)$. For $x\in\T\setminus\{\rho\},$ we let,
\[m(x^-)=\underset{t\to ht(x)^-}\lim m(\phi_{\rho,x}(t))\]
be the \emph{left limit} of $m$ at $x$. Similarly, we let
	\[\sum m(x^+)= \underset{i\in S}\sum \lim_{t\to ht(x)^+} m(\phi_{\rho,x_i}(t))
\]
be the \emph{additive right limit} of $m$ at $x$, where $(\T_i,i\in S)$ are the connected components of $\T_x\setminus\{x\}$ ($S$ being a countable index set), and $x_i$ being any point of $\T_i$ for $i$ in $S$. The following was then proven in \cite{Steph13}:

\begin{prop}\label{01} Assume that, for all $x\in\T$, $m(x^-)=m(x)\geq \sum m(x^+)$. Then there exists a unique Borel measure $\mu$ on $\T$ such that
	\[\forall x\in \T, \mu(\T_x)= m(x).
\]
\end{prop}

\subsection{The fragmentation tree}\label{buildtree}
We can build a tree which represents the genealogy of $\Pi^{(\alpha)}$, as was done originally in \cite{HM04} in the monotype and conservative case. The idea is that the lifetime of each integer $n$ is represented by a segment with length equal to the time it takes for this integer to be in a singleton, and for two different integers $n$ and $m$, these segments coincide up to a height equal to the time $t$ at which the blocks $\Pi_{(n)}(t)$ and $\Pi_{(m)}(t)$ split off. We formalise this with this proposition:

\begin{prop} There exists a unique compact rooted $\R$-tree $(\T,\rho,\mu)$ equipped with a set of points $(Q_n)_{n\in\N}$ such that:

\begin{itemize}
\item For all $n$, $ht(Q_n)=\inf\{t\geq0 : \; \{n\}\text{ is a block of }\Pi^{(\alpha)}(t)\}.$
\item For all $n\neq m$, $ht(Q_n \wedge Q_m)=\inf\{t\geq0:\; n\notin\Pi^{(\alpha)}_{(m)}(t)\}.$
\item The set $\underset{n\in\N}\bigcup\llbracket\rho,Q_n\rrbracket$ is dense in $\T$.
%\item If $x$ is a point of height $t\geq0$ with $x\in \llbracket\rho,Q_k\rrbracket$ then $e(x)$ is equal to the type of $\Pi_{(k)}(t^-)$.
\end{itemize}

\end{prop}

The construction and proof of uniqueness of $\T$ is fairly elementary and identical to the one in the monotype case, and we refer the interested reader to sections 3.2 and 3.3 of \cite{Steph13}. We will just focus on compactness here.

\begin{lem}\label{precompact} For $t\geq 0$ and $\veps>0$, let $N_t^{\veps}$ be the number of blocks of $\Pi^{(\alpha)}(t)$ which are not completely reduced to singletons by time $t+\veps$. Then $N_t^{\veps}$ is finite a.s.
\end{lem}

\begin{proof} For all $n\in\N$, let $\zeta_n=\inf\{s\geq0:\, \Pi^{(\alpha)}(t)\cap \Pi^{(\alpha)}_n(t+s)\text{ is made of singletons}\}$. By self-similarity, conditionally on $\F^{\overline{\Pi}^{(\alpha)}}_t$, $\zeta_n$ has the same distribution as $|\Pi_n^{(\alpha)}(t)|^{-\alpha}\tilde{\zeta}$, where $\tilde{\zeta}$ is an independent copy of $\zeta,$ under $\pr_{i^{(\alpha)}_n(t)}.$ By Proposition \ref{deathtimeexpmoments}, we know that there exist two constants $A>0$ and $B>0$ such that, for all $j\in[K]$ and $t\geq 0$
\[\pr_j[\zeta>t]\leq Ae^{-Bt}\] 
We can then bound the conditional expectation of $N_t^{\veps}$:
\begin{align*}
\E_i[N_t^{\veps}\mid\F^{\overline{\Pi}^{(\alpha)}}_t] &= \E_i \Big[\sum_{n\in\N} \mathbbm{1}_{\{|\Pi^{(\alpha)}_n(t)|^{-\alpha}\tilde{\zeta}>\veps\}}\mid \F^{\overline{\Pi}^{(\alpha)}}_t\Big] \\
                          &\leq \sum_{n\in\N} \pr_{i_n(t)}\big[\tilde{\zeta}>\veps|\Pi^{(\alpha)}_n(t)|^{\alpha}\mid \F^{\overline{\Pi}^{(\alpha)}}_t\big] \\
                          &\leq A\sum_{n:|\Pi^{(\alpha)}_n(t)|>0} e^{-B\veps|\Pi^{(\alpha)}_n(t)|^{\alpha}}.
\end{align*}
Letting $C=\underset{x>0}\sup\, Ax^{-1/\alpha}e^{-Bx}$, which is finite since $\alpha<0$, we have
\[\E_i[N_t^{\veps}\mid\F_t]\leq C\veps^{1/\alpha}\sum_n |\Pi_n^{(\alpha)}(t)|\leq C\veps^{1/\alpha},\]
which implies that $N_t^{\veps}$ is a.s. finite.
\end{proof}

\noindent \textit{Proof that $\T$ is compact.} We follow the idea of the proof of \cite[Lemma 5]{HM04}. Let $\veps>0,$ we will provide a finite covering of the set $\{Q_n,n\in\N\}$ by balls of radius $4\veps$, of which the compactness of $\T$ follows. For $n\in\N$, take $k\in\Z_+$ such that $k\veps <ht(Q_n)\leq (k+1)\veps.$ Then, for any $m$ such that $k\veps <ht(Q_m)\leq (k+1)\veps$ and $m\in\Pi_{(n)}\big((k-1)\veps\vee 0\big)$, we have $d(Q_n,Q_m)\leq 4\veps.$ This lets us define our covering : for $k\in\Z_+,$ consider the set
\[B_k^{\veps}=\{n\in\N:k\veps<ht(Q_n)\leq (k+1)\veps\}\]
of integers which are not yet in a singleton by time $k\veps$, but which are by time $(k+1)\veps.$ By Lemma \ref{precompact}, we know that, for $k\geq 1$, the number of blocks of $\Pi^{(\alpha)}((k-1)\veps))\cap B_k^{\veps}$ is finite, and less than or equal to $N^{\veps}_{(k-1)\veps}.$ Considering one integer $m$ per such block, taking the ball of center $Q_m$ and radius $4\veps$ yields a covering of $\{Q_n,n\in B_k^{\veps}\}$. We then repeat this for all $k$ with $1\leq k \leq \zeta/\veps$ (noticing that $B_k^{\veps}$ is empty for higher $k$), and finally for $k=0$, add the ball centered at $Q_m$ for any $k\in B_0^{\veps}$ if it is nonempty.
\qed

\vspace{1cm}

For $k\in\N$ and $t\leq ht(Q_k)$, we let $Q_k(t)=\phi_{\rho,Q_k}(t)$ be the unique ancestor of $Q_k$ with height $t$.
\begin{prop} There exists a unique measure $\mu$ on $\T$ such that $(\T,\mu)$ is a measurable random compact measured $\R$-tree and, a.s., for all $n\in\N$ and $t\geq 0$,
\begin{equation}\label{mesuresimple}\mu(\T_{Q_k(t)})=|\Pi^{(\alpha)}_{(k)}(t^-)|.
\end{equation}
\end{prop}
\begin{proof} The existence of a measure which satisfies (\ref{mesuresimple}) is assured by Proposition \ref{01}. The fact that $(\T,\mu)$ is then measurable for the Borel $\sigma$-algebra associated to the Gromov-Hausdorff-Prokhorov topology comes from writing it as the limit of discretised versions, see \cite{Steph13}.
\end{proof}

\subsection{Consequences of the Malthusian hypothesis}
In this section we assume the existence of a Malthusian exponent $p^*$, as well as the stronger assumption \eqref{Mq} for some $q>1$.

\subsubsection{A new measure on $\T$}

For all $n\in\N$ and $t,s\geq0$, let
\[M_{n,t}(s)=\frac{1}{b_i}\sum_{m:\Pi_m(t+s)\subset\Pi_{(n)}(t)}b_{i_m(t+s)}|\Pi_m(t+s)|^{p^*}.\]
By the Markov and fragmentation properties, we now that, conditionally on $\F^{\overline{\Pi}}_t,$ $\overline{\Pi}(t+\cdot)\cap \Pi_{(n)}(t)$ is a homogeneous fragmentation of the block $\overline{\Pi}_{(n)}(t)$ with the same characteristics $\big(\alpha,(c_i)_{i\in[K]},(\nu_i)_{i\in[K]}\big).$
With this point of view, the process $M_{n,t}(\cdot)$ is, its additive martingale, multiplied by the $\F^{\overline{\Pi}}_t$-measurable constant $\frac{b_{i_n(t)}}{b_i}|\Pi_n(t)|^{p^*}$. As such it converges a.s. and in $L^q$ to a limit $W_{n,t}$. By monotonicity, we can also define the left limit $W_{n,t^-}$.

\begin{prop} On an event with probability one, $W_{n,t}$ and $W_{n,t^-}$ exist for all $n\in\N$ and $t\geq 0$, and there exists a.s. a unique measure $\mu^*$ on $\T$, fully supported by the leaves of $\T$, such that, for all $n\in\N$ and $t\geq 0$,
\[\mu^*(\T_{Q_n(t)})=W_{n,\tau_n^{(-\alpha)}(t)^-} \]
\end{prop}

This is proved as Theorem 4.1 of \cite{Steph13} in the monotype case, and the same proof applies to our case without modifications, so we do not reproduce it here.

\medskip

Note that the total mass of $\mu^*$, which is the limit of the additive martingale, is not necessarily $1$, but its expectation is equal to $1$. Thus we can use it to create new probability distributions.

\subsubsection{Marking a point with $\mu^*$}
It was shown in \cite{Steph13} that, in the monotype case, the measure $\mu^*$ is intimately linked with the biasing described in Section \ref{biasing}. As expected, this also generalises here.

\begin{prop}\label{marking} For a leaf $L$ of the fragmentation tree $\T$ and $t\geq 0$, let $\overline{\Pi}^{(\alpha)}_L(t)$ be the $K$-type partition such that $(\N\setminus\{1\})\cap \overline{\Pi}^{(\alpha)}_L(t) = (\N\setminus\{1\})\cap \overline{\Pi}^{(\alpha)}(t)$, and $1$ is in the same block as an integer $n$ if and only if $ht(L\wedge Q_n)> t$. Then let $\overline{\Pi}_L(t)$ be the partition such that $(\N\setminus\{1\})\cap \overline{\Pi}_L(t) = (\N\setminus\{1\})\cap \overline{\Pi}(t),$ and $1$ is put in the block of any $n$ such that $1$ is also in $(\Pi_L^{\alpha})_{(n)}(r)$ with $t=\tau^{(\alpha)}_{n}(r)$.

We then have, for any non-negative measurable function $F$ on $\mathcal{D}$,
\[\E_i\left[\int_\T F(\overline{\Pi}_L)\mathrm d\mu^*(L)\right]=\E_i^*[F(\Pi)],\]
where the measure $\pr_i^*$ was defined in Section \ref{biasing}.
\end{prop}

\begin{proof}
Assume first that the function $F$ can be written as $F(\bar{\pi})=K\big(\bar{\pi}(s),0\leq s\leq t\big),$ for a certain $t\geq0$ and $K$ a function on $\mathcal{D}_t.$ For $n\in\N$ and $s\leq t$, let $\overline{\Psi}^n(s)$ be the same partition as $\overline{\Pi}(s)$, except that $1$ is put in the same block as any integer $m$ with $m\in\Pi_n(t).$ We can then write

\[\int_\T F(\overline{\Pi}_L)\mathrm d\mu^*(L)=\sum_n W_{n,t}\,F\big(\overline{\Psi}^n(s),0\leq s\leq t)\big).\]

Recall that we can write $W_{n,t}=\frac{b_{i_n(t)}}{b_i}|\Pi_n(t)|^{p^*}X_{n,t},$ where, conditionally on $i_n(t)$, $X_{n,t}$ is the limit of the additive martingale for an independent version of the process under $\pr_{i_n(t)}$. Hence for any $j\in [K]$, $\E_i[X_{n,t}\mid \mathcal{F}_t,i_n(t)=j]=1$, implying $\E_i[X_{n,t}\mid \mathcal{F}_t]=1$, and thus

\begin{align*}
\E_i\Big[\int_\T F(\overline{\Pi}_L)\mathrm d\mu^*(L)\Big]	&=\E\Big[\sum_n \frac{b_{i_n(t)}}{b_i}|\Pi_n(t)|^{p^*}X_{n,t}\,F\big(\overline{\Psi}^n(s),0\leq s\leq t)\big)\Big] \\
&=\frac{1}{b_i}\E\Big[\sum_n b_{i_n(t)}|\Pi_n(t)|^{p^*}F\big(\overline{\Psi}^n(s),0\leq s\leq t\big)\Big] \\
&=\E_i^{\bullet}\Big[F\big(\overline{\Psi}^n(s),0\leq s\leq t\big)\Big] \\
&=\E_i^{*}\Big[F\big(\overline{\Psi}^n(s),0\leq s\leq t\big)\Big].
\end{align*}

A measure theory argument then extend this to any $\mathcal{D}$-measurable function $F$, as done in the proof of Proposition 5.3 in \cite{Steph13}.

\end{proof}

\begin{cor}\label{markingmoment} For any $p\in\R$, we have 
\[\E_i\big[\int_{\T} (ht(L))^{p}\mathrm d\mu^*(L)\big]=\E_i[I^p_{|\alpha|\xi}],\]
where $\big(\xi_t,J_t),t\geq0\big)$ a MAP with Bernstein matrix $\mathbf{\Phi}^*$ and $I_{|\alpha|\xi}$ is the exponential functional of $|\alpha|\xi.$
\end{cor}
\begin{proof} Apply Proposition \ref{marking} to the function $F$ defined by
\[F(\bar{\pi})=\Big(\inf\{t\geq0:\,|\pi_1(t)|=0\}\Big)^{p}.\]
Recalling that, under $\pr_i^*$, $\big(|\Pi_1(t),i_1(t)),t\geq0\big)$ is the $\alpha$-self-similar Lamperti transform of a MAP with Bernstein matrix $\mathbf{\Phi}^*.$ We then know from Section \ref{LampertiMAP} that its death time has the distribution of $I_{|\alpha|\xi}$, ending the proof.
\end{proof}
\subsubsection{The biased tree}\label{sec:biasedtree}

We give here a few properties of the tree built from $\overline{\Pi}$ under the distribution $\pr_i^*$.

\begin{itemize}
\item The spine decomposition obtained at the end of Section \ref{biasing} helps give a simple description of the tree. Keeping in line with the Poisson point process notation from that section, as well as the time-changes $\tau_n^{(\alpha)}$ for $n\in\N$, the tree is first made of a spine, which represents the lifetime of the integer $1$, and has length $(\tau_1^{(\alpha)})^{-1}(\infty).$ The leaf at the edge of this segment is the point $Q_1$ from Section \ref{buildtree}. On this spine are then attached many rescaled independent copies of $\T$. Specifically, for $t>0$ such that $|\Pi^{(\alpha)}_1(t)|<|\Pi^{(\alpha)}_1(t^-)|$, the point of height $t$ of the spine is (usually) a branchpoint: for all $n\geq 2$ such that $|\big(\Delta^{(1,i_1^{(\alpha)}(t^-))}(t)\big)_n|\neq 0$, we graft a subtree $\T'_{n,t}$ which can be written as $\Big(|\Pi^{(\alpha)}_1(t^-)||\big(\Delta^{(1,i_1(t^-))}(\tau_1(t))\big)_n|\Big)^{-\alpha}\T'$, where $\T'$ is an independent copy of $\T$ under $\pr_{\delta^{(1,i_1(t^-))}_n(\tau_1(t))}$.
\item Under $\pr_i^*$, $\T$ is still compact. This is because the result of Lemma \ref{precompact} still holds: of all the blocks of $\Pi^{(\alpha)}$ present at a time $t$, only the one containing the integer $1$ will behave different from the case of a regular fragmentation process, and so all but a finite number of them will have been completely reduced to dust by time $t+\veps$ a.s. for a $\veps>0$. From this, the proof of compactness is identical.
\item We can use the spine decomposition to define $\mu^*.$ For each pair $(t,n)$ such that $\T'_{n,t}$ is grafted on the spine, the subtree comes with a measure $\mu^*_{n,t}$ which can be written as $|\Pi_1(t^-)|^{p^*}|\big(\Delta^{(1,i_1(t^-))}(t)\big)_n|^{p^*} (\mu^*)'$, where $(\mu^*)'$ is an independent copy of $\mu^*$ under $\pr_{\delta^{(1,i_1(t^-))}_n(\tau_1(t))}.$ We then let 
\[\mu^*=\sum_{n,t}\mu^*_{n,t}.\]
\end{itemize}

\subsubsection{Marking two points}
We will be interested in knowing what happens when we mark \emph{two} points ``independently" with $\mu^*$, specifically we care about the distribution of the variable
\[\int_{\T}\int_{\T} F(\T,L,L')\mathrm d\mu^*(L)\mathrm d\mu^*(L'),\]
where $F$ is a nonnegative measurable function on the space of compact, rooted, measured and 2-pointed trees (equipped with an adapted GHP metric - see for example \cite{M09}, Section 6.4).

The next proposition shows that, in a sense, marking two leaves with $\mu^*$ under $\pr_i$ is equivalent to taking the tree under $\pr_i^*$ and marking the leaf at the end of the spine as well as another chosen according to $\mu^*.$

\begin{prop}\label{doublemarking} We have
\[\E_i\Big[\int_{\T}\int_{\T} F(\T,L,L')\mathrm d\mu^*(L)\mathrm d\mu^*(L')\Big]=\E_i^*\Big[\int_{\T} F(\T,Q_1,L')\mathrm d\mu^*(L')\Big]\]
\end{prop}
\begin{proof}
Start by defining the processes $\overline{\Pi}_L^{(\alpha)}$ and $\Pi_L$ under $\pr_i$, as in the proof of Proposition \ref{marking}. We know that $\overline{\Pi}_L$ fully encodes $\T$ and $L$, and with a little extra information, it can also encode the other leaf $L':$ for all $t\leq ht(L')$, let $n_{L'}^{(\alpha)}(t)$ be the smallest $n\neq 1$ such that $L'\in \T_{Q_n(t)},$ and for any $t\geq0$, $n_{L'}(t)=n^{(\alpha)}_{L'}\big((\tau_{n^{(\alpha)}(t)}^{(\alpha)})^{-1}(t)\big).$ Then $(\T,L,L')$ is the image of $\big((\overline{\Pi}_L(t),n_{L'}(t)),t\geq0\big)$ by a measurable function.

Thus, up to renaming functions, we are reduced to proving that
\[\E_i\Big[\int_{\T}\int_{\T} F\big((\overline{\Pi}_L(t),n_{L'}(t)),t\geq0\big)\mathrm d\mu^*(L)\mathrm d\mu^*(L')\Big]=\E_i^*\Big[\int_{\T} F\big((\overline{\Pi},n_{L'}(t)),t\geq0\big)\mathrm d\mu^*(L')\Big]\]

From there we can proceed similarly as in the proof of Proposition \ref{marking}. Assume that $F\big((\bar{\pi}(s),n(s)),s\geq0\big)$ can be written as $K\big((\bar{\pi}(s),n(s)),s\leq t\big)$ for some $t\geq 0$ and a measurable function $K$ on the appropriate space, then we split the integral with respect to $\mathrm d\mu^*(L')$ according to which block of $\Pi_L(t)$ the integer $n_{L'}(t)$ is in:
\[\int_{\T}\int_{\T} F\big((\overline{\Pi}_L(s),n_{L'}(s)),s\geq0\big)\mathrm d\mu^*(L)\mathrm d\mu^*(L')=\int_{\T}\sum_{n\in\N} W_{n(t),t}K\big((\overline{\Pi}_L(s),n(s)),s\leq t\big)\mathrm d\mu^*(L).\]

In the right-hand side, $n(s)$ is defined as the smallest integer of the block of $\overline{\Pi}_L(s)$ which contains the $n$-th block of $\Pi_L(t)$. Now, Proposition \ref{marking} tells us that the expectation of the right-hand side is equal to
\[\E_i^*\Big[\sum_{n\in\N} W_{n(t),t}K\big((\overline{\Pi}(s),n(s)),s\leq t\big)\Big],\]
and hence is also equal to
\[\E_i^*\Big[\int_{\T}K\big((\overline{\Pi}(s),n_{L'}(s)),s\leq t\big)\mathrm d\mu^*(L')\Big],\]
which is what we wanted. Another measure theory argument then generalizes this to all functions $F$.
\end{proof}
\section{Hausdorff dimension of $\T$}

Let $(M,d)$ be a compact metric space. For $F\subset M$ and $\gamma>0$, we let
\[m_{\gamma}(F)=\underset{\veps>0}\sup \inf \sum_{i\in I} \mathrm{diam}(E_i)^\gamma,\]
where the infimum is taken over all the finite or countable coverings $(E_i,i\in I)$ of $F$ by subsets with diameter at most $\veps.$ The Hausdorff dimension of $F$ can then be defined as
\[\dim_{\mathcal{H}}(F)=\inf\{\gamma>0:\, m_{\gamma}(F)=0\}=\sup\{\gamma>0:\, m_{\gamma}(F)=\infty\}.\]
We refer to \cite{Falconer} for more background on the topic.

The aim of this section is to establish the following theorem, which gives the exact Hausdorff dimension of the set of leaves of the fragmentation tree, which we call $\mathcal{L}(\T).$

\begin{thm}\label{dimension} Assume that there exists $p\in[0,1]$ such that $\lambda(p)<0$. Then there exists a Malthusian exponent $p^*$ and, a.s., if $\overline{\Pi}$ does not die in finite time, then
\[\dim_{\mathcal{H}}(\mathcal{L}(T))=\frac{p^*}{|\alpha|}.\]
\end{thm}
We recall that, in the conservative cases where $c_i=0$ for all $i$ and $\nu_i$ preserves total mass for all $i$, we have $p^*=1$ and so the dimension is $\frac{1}{|\alpha|}.$

The proof of Theorem \ref{dimension} will be split in three parts: first we show that $\dim_{\mathcal{H}}(\mathcal{L}(T))$ is upper-bounded by $\frac{p^*}{|\alpha|}$ , then we show the lower bound in some simpler cases, and finally get the general case by approximation.
\subsection{Upper bound}
Recall that, for $p>0$, we have defined $\lambda(p)=-\lambda(-\mathbf{\Phi}(p-1))$ and that it is a strictly increasing and continuous function of $p$. The following lemma then implies the upper-bound part of Theorem \ref{dimension}.
\begin{prop}
Let $p>0$ such that $\lambda(p)>0$. Then we have, a.s.,
\[\dim_{\mathcal{H}}(\mathcal{L}(T))\leq\frac{p}{|\alpha|}.\]
\end{prop}

\begin{proof} We will exhibit a covering of the set of leaves by small balls such that the sum of the $\frac{p}{|\alpha|}$-th powers of their radiuses has bounded expectation as the covering gets finer. Fix $\veps>0$, and for $n\in\N$, let
	\[t_n^{\veps} = \inf \{t\geq0: |\Pi_{(n)}^{(\alpha)}(t)|< \veps\}.
\]
We use these times to define another exchangeable partition $\overline{\Pi}^{\veps}$, such that the block of $\overline{\Pi}^{\veps}$ containing an integer $n$ is $\overline{\Pi}_{(n)}^{(\alpha)}(t_n^{\veps}).$ Consider also, still for an integer $n$, the time
\[\zeta_n^{\veps}=\inf\{t\geq0:\, \Pi^{\veps}_{(n)}\cap \Pi^{\alpha}(t_n^{\veps}+t)\text{ is made of singletons}\}.\]
We can now define our covering: for one integer $n$ per block of $\Pi^{\veps}$, take a closed ball centered at point $Q_n(t_n^{\veps})$ and with radius $\zeta_n^{\veps}$.

Let us check that this indeed a covering of the leaves of $\T$. Let $L$ be a leaf, and, for $t<ht(L)$, let $n(t)$ be the smallest integer $n$ such that the point of height $t$ of the segment $[0,L]$ is $Q_n(t)$. If $L=Q_n$ for some $n$ then $n(t)$ is eventually constant, and then $L$ is trivially in the ball centered at $Q_n(t_n^{\veps})$ with radius $\zeta_n^{\veps}$. If not, then $n(t)$ tends to infinity as $t$ tends to $ht(L)$, and $|\Pi_{(n(t))}(t)|$ reaches $0$ continuously. Thus we take the first time $t$ such that $|\Pi_{(n(t))}(t)|<\veps,$ then $t=t^{\veps}_{n(t)}$ and $L$ is in the ball centered at $Q_{n(t)}(t)$ with radius $\zeta^{\veps}_{n(t)}$.

The covering is also \emph{fine} in the sense that $\sup_n \zeta_n^{\veps}$ goes to $0$ as $\veps$ goes to $0$. Indeed, if that wasn't the case, one would have a sequence $(n_l)_{l\in\N}$ and a positive number $\eta$ such that $\zeta^{2^{-l}}_{n_l}\geq\eta$ for all $n$. By compactness, one could then take a limit point $x$ of the sequence $(Q_{n_l}(t_{n_l}^{2^{-l}}) )_{l\in\N}$. $x$ would not be a leaf (by compactness, the subtree rooted at $x$ has height at least $\eta$), so we would have $x=Q_m(t)$ for some $m\in\N$ and $t<ht(Q_m),$ hence $|\Pi_{(m)}^{(\alpha)}(t)|>0,$ a contradiction since $|\Pi_{(n_l)}^{(\alpha)}(t_{n_l}^{2^{-l}})|$ tends to $0$.

By the extended fragmentation property at the stopping line $(t_n^{\veps},n\in\N)$, conditionally on $\overline{\Pi}^{\veps},$ the various $\zeta_n^{\veps}$ are independent, and for each $n$, $\zeta_n^{\veps}$ is equal in distribution to $|\Pi^{\veps}_{(n)}|^{|\alpha|}$ times an independent copy of $\zeta$ (under $\pr_{i_{(n)}^{\veps}}$). Thus we can write, summing in the following only one integer $n$ per block of $\Pi^{\veps}$,
\begin{align*}
\E_i\left[\sum_n (\zeta_{(n)}^{\veps} )^{\frac{p}{|\alpha|}}\right] &\leq \E_i\left[\sum_n \E_{i_{n}^{\veps}}\big[\zeta^{p/|\alpha|}\big]|\Pi^{\veps}_{(n)}|^{p} \right]  \\																										 &\leq \underset{j\in[K]}\sup\E_j\big[\zeta^{p/|\alpha|}\big] \E_i\left[\sum_n  |\Pi^{\veps}_{(n)}|^{p}\right].
\end{align*}

We know from Proposition \ref{deathtimeexpmoments} that $\underset{j\in[K]}\sup\E_j\big[\zeta^{p/|\alpha|}\big]$ is finite, so we only need to check that the other factor is bounded as $\veps$ tends to $0$. Since $\Pi^{\veps}$ is exchangeable, we have
\begin{align*} 
\E_i\left[\sum_n |\Pi^{\epsilon}_n|^{p}\right]         &= \E_i\big[|\Pi^{\epsilon}_1|^{p-1}\mathbf{1}_{\{|\Pi^{\epsilon}_1|\neq 0\}}\big] \\
																														 &= \E_i\big[|\Pi_1(T_\epsilon)|^{p-1}\mathbf{1}_{\{|\Pi_1(T_\epsilon)|\neq 0\}}\big]\\
																														 &\leq \E_i\big[|\Pi_1(T_0^-)|^{p-1}\big],
\end{align*}
where $T_{\epsilon}=\inf\{t,|\Pi_1(t)|\leq\epsilon\}$ and $T_0=\inf\{t,|\Pi_1(t)|=0\}$. We have thus reduced our problem to a question about moments of a MAP - recall that $|\Pi_1(T_0^-)|=e^{-\xi_{T^-}}$, where $\big((\xi_t,J_t),t\geq 0\big)$ is a MAP with Bernstein matrix $\mathbf{\Phi}$ defined in (\ref{tagmatrix}), and $T$ is its death time.
Proposition \ref{deathmoment} then says that, for $p$ such that $-\lambda(-\mathbf{\Phi}(p-1))>0$, i.e. such that $\lambda(p)>0$, $\E_i[e^{-\xi_{T^-}}]$ is finite, and this ends our proof.

\end{proof}
\subsection{The lower bound in a simpler case}
We prove the lower bound for dislocation measures such that splittings occur at finite rates, and splittings are at most $N$-ary for some $N\in\N.$
\begin{prop}\label{lowerboundeasy} Assume that:
\begin{itemize}
\item The fragmentation is Malthusian, with Malthusian exponent $p^*.$
\item For all $i\in [K]$, $\nu_i\Big(\big\{s_2>0\big\}\Big)<\infty.$ 
\item There exists $N\in\N$ such that, for all $i\in [K]$, $\nu_i\Big(\big\{s_{N+1}>0\big\}\Big)=0.$
\end{itemize} 
\eqref{Mq} is then automatically satisfied for all $q>1$. Moreover, a.s., if $\overline{\Pi}$ does not die in finite time, we have
\[\dim_{\mathcal{H}}(\mathcal{L}(T))\geq\frac{p^*}{|\alpha|}\]
\end{prop}

\begin{proof}
Before doing the main part of the proof, let us check \eqref{Mq}: that $\Big|1-\sum_{1}^Ns_i^{p^*}\Big|^q$ is $\nu_i$-integrable for all $i$. Write
\[1-\sum_{n=1}^Ns_n^{p^*} \leq 1-s_1^{p^*}\leq C_{p^*}(1-s_1)\]
where $C_{p^*}=\underset{x\in[0,1)}\sup{\frac{1-x^{p^*}}{1-x}},$ and also
\[1-\sum_{n=1}^Ns_n^{p^*}\geq (1-N)\mathbbm{1}_{\{s_2>0\}}.\]
This gives us an upper and a lower bound of $1-\sum_{1}^Ns_n^{p^*},$ and so we can write
\[\Big|1-\sum_{1}^Ns_n^{p^*}\Big|^q\leq C_{p^*}^q(1-s_1)^q+(N-1)^q\mathbbm{1}_{\{s_2>0\}}\leq C_{p^*}^q(	1-s_1)+(N-1)^q\mathbbm{1}_{\{s_2>0\}},\]
and this is $\nu_i$-integrable for all $i\in[K]$, by assumption. (note that $1-s_1\leq 1-s_1\mathbf{1}_ {\{i_1=i\}}$)

\bigskip

Now for the lower bound on the Hausdorff dimension. We want to use Frostman's lemma (\cite[Theorem 4.13]{Falconer}) for the measure $\mu^*$: we will show that, for $\gamma<\frac{p^*}{|\alpha|}$,
\[\E_i\Big[\int_{\T}\int_{\T}d(L,L')^{-\gamma}\mathrm d\mu^*(L)\mathrm d\mu^*(L')\Big]<\infty,\]
which does imply that, on the event where $\mu^*$ is not the zero measure (which is the event where $\overline{\Pi}$ does not die in finite time), the Hausdorff dimension of the support of $\mu^*$ is larger than $\frac{p^*}{|\alpha|}.$

By Proposition \ref{doublemarking}, we have
\[\E_i\Big[\int_{\T}\int_{\T}d(L,L')^{-\gamma}\mathrm d\mu^*(L)\mathrm d\mu^*(L')\Big]=\E_i^*\Big[\int_{\T}d(Q_1,L)^{-\gamma}\mathrm d\mu^*(L)\Big].\]

We can give an upper bound the right-hand side of this equation by using the spine decomposition of $\T$ under $\pr_i^*$ given in Section \ref{sec:biasedtree}: for appropriate $n\geq 2$ and $t>0$, $\T'_{n,t}$ is the $n$-th tree attached to the spine at the point $Q_1(t)$. If we let 
\[Z_{n,t}=\int_{\T'_{n,t}}\frac{d(L,Q_1(t))^{-\gamma}}{\Big(\big|\Pi_1(\tau_1(t)^-)\big|\,\big|\Delta^{(1,i_1(\tau(t)^-))}_n(\tau_1(t))\big|\Big)^{p^*+\alpha\gamma}}\mathrm d\mu^*(L),\] we then have

\begin{align*}
\E_i^*\Big[\int_{\T}d(Q_1,L)^{-\gamma}\mathrm d\mu^*(L)\Big]&=\E_i^*\Big[\sum_{t\geq 0}\sum_{n\geq 2}\int_{T'_{n,t}}d(Q_1,L)^{-\gamma}\mathrm d\mu^*(L)\Big] \\
     &\leq \E_i^*\Big[\sum_{t\geq 0}\sum_{n\geq 2}\int_{T'_{n,t}}d(Q_1(t),L)^{-\gamma}\mathrm d\mu^*(L)\Big] \\
     &=\E_i^*\Big[\sum_{t\geq 0}\sum_{n\geq 2} \big(|\Pi_1(t^-)||\Delta_n(t)|\big)^{p^*+\alpha\gamma} Z_{n,t}\Big].
\end{align*}
Notice then that, by the fragmentation property, conditionally on $\mathcal{\F}^{\overline{\Pi}^{(\alpha)}}_t$, $Z_{n,t}$ has the same distribution as $\int_{\T} ht(L)^{-\gamma}\mathrm d\mu^*(\gamma)$ under $\pr_j$, where $j=\delta_n^{(1,i_1(t^-))}(t)$. This is why we extend, for all $j\in[K]$, the Poisson point processes $(\overline{\Delta}^{(1,j)}(t),t\geq 0)$ into $\Big(\big(\overline{\Delta}^{(1,j)}(t),(Y_{n,t}^{(j,k)})_{(k,n)\in[K]\times\{2,3,\ldots\}}\big),t\geq 0\Big)$, where, conditionally on $\overline{\Delta}^{(1,j)}(t)$, the $\big(Y_{n,t}^{(j,k)}\big)_{(k,n)\in[K]\times\{2,3,\ldots\}}$ are independent, and for each $k$ and $n$, $Y_{n,t}^{(j,k)}$ has the distribution of the exponential function $I_{|\alpha|\xi}$ of $|\alpha|\xi$ where $(\xi,J)$ is a MAP starting at $(0,k)$ with Bernstein matrix $\mathbf{\Phi}^*$. We can then write
\[\E_i^*\Big[\int_{\T}d(Q_1,L)^{-\gamma}\mathrm d\mu^*(L)\Big]\leq\E_i^*\Big[\sum_{t\geq 0}\sum_{n\geq 2} \big(|\Pi_1(t^-)||\Delta_n(t)|\big)^{p^*+\alpha\gamma} Y^{(i_1(t^-),\delta_n(t))}_{n,t}\Big].\]
Having rewritten this in terms of a Poisson point process, and since $|\Pi_1(t^-)|$ and $i_1(t^-)$ are predictable, we can directly apply the Master Formula:
\begin{align*}
\E_i^*\Big[\int_{\T}&d(Q_1,L)^{-\gamma}\mathrm d\mu^*(L)\Big]\leq\E_i\Big[\int_{0}^{\infty}\mathrm dt\int_{\s}|\Pi_1(t^-)|^{p^*+\alpha\gamma}\sum_{n\in\N}b_{i_n}s_n^{p^*}\sum_{m\neq n}s_m^{p^*+\alpha\gamma}\E_{i_m}[I_{|\alpha|\xi}^{-\gamma}]\mathrm d\nu_{i_1(t^-)}(\bar{\mathbf{s}}) \Big] \\
&\leq \E_i\Big[\int_{0}^{\infty}|\Pi_1(t^-)|^{p^*+\alpha\gamma}\mathrm dt\Big]\underset{j\in[K]}\sup \E_{j}[I_{|\alpha|\xi}^{-\gamma}]\underset{j\in[K]}\sup\int_{\s}\sum_{n\in\N}b_{i_n}s_n^{p^*}\sum_{m\neq n}s_m^{p^*+\alpha\gamma}\mathrm d\nu_j(\bar{\mathbf{s}} ) \\
&\leq \underset{j\in[K]}\sup b_j \underset{j\in[K]}\sup \E_j\Big[\int_0^{\infty} e^{-(p^*+\alpha\gamma)\xi_t}\mathrm d t\Big]\underset{j\in[K]}\sup \E_{j}[I_{|\alpha|\xi}^{-\gamma}]\underset{j\in[K]}\sup\int_{\s}\sum_{n\in\N}s_n^{p^*}\sum_{m\neq n}s_m^{p^*+\alpha\gamma}\mathrm d\nu_j(\bar{\mathbf{s}} ). \\
\end{align*}
All that is left is to check that all the factors are finite for $\gamma < \frac{p^*}{|\alpha|}$. Fix $j\in[K]$:
\begin{itemize}
\item By (\ref{eq:bernstein}), we have $\E_j[e^{-(p^*+\alpha\gamma)\xi_t}]=\sum_k \big(e^{-t\mathbf{\Phi}^*(p^*+\alpha\gamma)}\big)_{j,k}.$ Recalling from (\ref{phistar}) the definition of $\mathbf{\Phi}^*$, we see that the smallest real part of an eigenvalue of $\mathbf{\Phi}^*(q)$ is positive for $q>0$, thus for $\gamma<\frac{p^*}{|\alpha|}$, the matrix integral $\int_0^{\infty}e^{-t\mathbf{\Phi}^*(p^*+\alpha\gamma)}\mathrm dt$ is well defined, and $\E_j[e^{-(p^*+\alpha\gamma)\xi_t}]<\infty.$
\item Note that $\big((|\alpha|\xi_t,J_t),t\geq0\big)$ is a MAP with Berstein matrix $\mathbf{\Phi}^*(|\alpha|\cdot).$ Thus, by Proposition \ref{negativemoments}, $\E_{j}[I_{|\alpha|\xi}^{-\gamma}]$ will be finite if $|\alpha|(-\gamma+1)>\underline{p}+1-p^*$, where $\underline{p}=\inf\{p\in\R:\forall k,l, (\mathbf{\Phi}(p))_{k,l}>-\infty\}.$ However, with our assumptions that the dislocation measures are finite and $N$-ary, $\underline{p}\leq -1$. Indeed, for any $p>-1$, we can write, fixing $k\in[K],$
\begin{align*}
\Big(\sum_{l=1}^{K}\mathbf{\Phi}(p)\Big)_{k,l}&\geq \int_{\s}\big(1-\sum_{n=1}^N s_n^{1+p}\big) \mathrm d\nu_k(\mathbf{\bar{s}}) \\
                     &\geq -\int_{\s}\sum_{n=2}^N s_n^{1+p}\mathrm d\nu_k(\mathbf{\bar{s}}) \\
                     &\geq  -(N-1)\nu_k\Big(\big\{s_2>0\big\}\Big)\\
                     &>-\infty.
\end{align*}
The fact that $|\alpha|(-\gamma+1)>\underline{p}+1-p^*$ then follows readily from $\underline{p}\leq -1$ and $p^*+\alpha\gamma>0$.
\item The last factor works similarly: since $p^*+\alpha\gamma>0$, we can simply write
\[\int_{\s}\sum_{n\in\N}s_n^{p^*}\sum_{m\neq n}s_m^{p^*+\alpha\gamma}\mathrm d\nu_j(\bar{\mathbf{s}} ) \leq N^2\nu_j\Big(\big\{s_2>0\big\}\Big)<\infty,\]
which ends our proof.
\end{itemize}

\end{proof}

\subsection{General case of the lower bound by truncation}
Most families of dislocation measures satisfying the assumptions of Theorem \ref{dimension} do not satisfy the stronger ones of Proposition \ref{lowerboundeasy}, however a simple truncation procedure will allow us to bypass this problem. Fix $N\in\N$ and $\veps>0$, and let $\G_{N,\veps}:\s\mapsto\s$ be defined by
	\[G_{N,\epsilon}(\mathbf{s})= 
	\begin{cases}
	\big((s_1,i_1),\ldots,(s_N,i_N),(0,0),(0,0),\ldots\big) & \text{if } s_1\leq 1-\epsilon \\
   \big((s_1,i_1),(0,0),(0,0),\ldots\big) & \text{if } s_1>1-\epsilon.
\end{cases}
\]
Then if we let, for all $i\in[K]$, $\nu^{N,\veps}_i=(G^{N,\veps})_{*}\nu_i$ be the image measure of $\nu_i$ by $G^{N,\veps}$, then the $\big((c_i,\nu^{N,\veps}_i),i\in[K]\big),$ if Malthusian, satisfy the assumptions of Proposition \ref{lowerboundeasy}. To properly use this, we need some additional setup. First, we define a natural extension of $G^{N,\veps}$ to $\p_N.$ For $\bar{\pi}\in\p_{\N}$ which does not have asymptotic frequencies for all its blocks, let $G^{N,\veps}(\bar{\pi})=\bar{\pi}$ (this doesn't matter, this measurable event has measure $0$). Otherwise, call $\big((\pi_n^{\downarrow},i_n^{\downarrow}),n\in\N\big)$ the blocks of $\bar{\pi}$ with their types, ordered such that the asymptotic frequencies paired with the types are lexicographically decreasing (if there are ties, pick another arbitrary ordering rule, for example by least smallest element). Let then
	\[G^{N,\epsilon}(\pi)= 
	\begin{cases}
	\big((\pi_1^{\downarrow},i_1^{\downarrow})\ldots,(\pi_N^{\downarrow},i_N^{\downarrow}),\text{singletons}\big) & \text{if } |\pi_1^{\downarrow}|\leq 1-\epsilon \\
   \big((\pi_1^{\downarrow},i_1^{\downarrow}),\text{singletons}\big) & \text{if } |\pi_1^{\downarrow}|>1-\epsilon.
\end{cases}
\]

One can then easily couple a homogeneous fragmentation process $\overline{\Pi}$ with dislocation measures $(\nu_i,i\in[K])$ with a homogeneous fragmentation process $\overline{\Pi}^{N,\veps}$ with dislocation measures $(\nu^{N,\veps}_i,i\in[K])$: simply build the first one from Poisson point processes  $(\overline{\Delta}^{(n,j)}(t),t\geq0)$ (for $n\in\N$, $j\in[K]$) as usual, and the second one from the $G^{N,\veps}(\overline{\Delta}^{(n,j)}(t),t\geq0).$  Calling the respective $\alpha$-self-similar fragmentation trees $\T$ and $\T^{N,\veps}$, we clearly have $\T^{N,\veps}\subset \T$, $\mathcal{L}(\T^{N,\veps})\subset \mathcal{L}(\T)$ and even $\T^{N,\veps}\subset\T^{N',\veps'}$ for $N'\geq N,\veps'\leq \veps$. This implies in particular that $\dim_{\mathcal{H}}(\mathcal{L}(\T^{N,\veps}))\leq  \dim_{\mathcal{H}}(\mathcal{L}(\T)).$  Proving Theorem \ref{dimension} can then be done by establishing two small lemmas which show that the truncation procedure provides a good approximation.

Let $\mathbf{\Phi}_{N,\veps}(p)$ be the Bernstein matrix corresponding to the tagged fragment of $\Pi^{N,\veps}$:
\begin{align*}
\mathbf{\Phi}_{N,\veps}(p)=\big(c_i(p+1)\big)_{\mathrm{diag}}+&\left(\int_{\s}\left(\mathbbm{1}_{\{i=j\}}-\sum_{n=1}^{N}s_n^{1+p}\mathbbm{1}_{\{i_n=j\}}\right)\mathbbm{1}_{\{s_1\leq 1-\veps\}}\nu_i(\mathrm d \bar{\mathbf{s}})\right)_{i,j\in[K]}\\
   &+\left(\int_{\s}\Big(\mathbbm{1}_{\{i=j\}}-s_1^{1+p}\mathbbm{1}_{\{i_1=j\}}\Big)\mathbbm{1}_{\{s_1>1-\veps\}}\nu_i(\mathrm d \bar{\mathbf{s}})\right)_{i,j\in[K]}.
\end{align*}
It is straightforward to see that, for fixed $p$ this decreases with $N$, increases with $\veps$, and that its infimum (i.e. limit as $N$ goes to infinity and $\veps$ to $0$) is $\mathbf{\Phi}(p)$. By Proposition \ref{algebre}, if we let $\lambda_{N,\veps}(p)=-\lambda(-\mathbf{\Phi}_{N,\veps}(p-1))$, then $\lambda_{N,\veps}(p)$ also with $N$, decreases with $\veps$, and its supremum is $\lambda(p)$.
\begin{lem}
\begin{itemize}
\item[$(i)$]For $N$ large enough and $\veps$ small enough, $\Pi^{N,\veps}$ is Malthusian, and we call its Malthusian exponent $p^*_{N,\veps}$.
\item[$(ii)$] $p_{N,\veps}^*$ is an increasing function of $N$ and a decreasing function of $\veps$.
\item[$(iii)$] $p^*=\underset{\begin{subarray}{c}
  N\in\N \\
  \veps>0
  \end{subarray}}\sup p^*_{N,\veps}$
\end{itemize}
\end{lem}

\begin{proof}
For $(i)$, take $p$ such that $\lambda(p)<0$, which exists by the main assumption of Theorem \ref{dimension}. Then, for $N$ large enough and $\veps$ small enough, we have $\lambda_{N,\veps}(p)<0$. Since $\lambda_{N,\veps}(1)\geq 0$ (a fact which is true for any fragmentation), continuity of the eigenvalue guarantees that there exists $p^*_{n,\veps}$ such that $\lambda_{n,\veps}(p^*_{N,\veps})=0.$

For $(ii)$, take $N'\geq N$ and $\veps'\leq \veps$, we have by $(v)$ of Proposition \ref{algebre} $\lambda_{N',\veps'}(p^*_{N,\veps})\leq 0$, hence $p^*_{N',\veps'}\geq p^*_{N,\veps}$.

To prove $(iii)$, take $p<p^*,$ then since $\lambda_{N,1/N}(p)$ converges to $\lambda(p)<0$, we have $\lambda_{N,\veps}(p)<0$ for $N$ large enough and $\veps$ small enough, implying $p<p^*_{N,\veps}$. This shows that  $p^*= \underset{\begin{subarray}{c}
  N\in\N \\
  \veps>0
  \end{subarray}}\sup p^*_{N,\veps}.$
\end{proof}

\begin{lem} Almost surely, if $\Pi$ does not die in finite time, then, for $N$ large enough and $\veps$ small enough, the same holds for $\Pi^{N,\veps}.$
\end{lem}
\begin{proof} For $i\in[K]$, $N\in\N$ and $\veps>0$, let $q^{(i)}_{N,\veps}$ be the probability that $\Pi^{N,\veps}$ reduces to dust in finite time when starting from type $i$, and let $q^{(i)}$ be the same for $\Pi$. Showing that $q^{(i)}_{N,1/N}$ converges to $q^{(i)}$ will prove the lemma.

As with Lemma \ref{extinction}, this is a basic result on Galton-Watson processes which easily extends to the multi-type setting. For $N\in\N$, $\veps>0,$ $n\in\Z_+$ and $j\in[K]$, let $Z_{N,\veps}^{(j)}(n)$ be the number of blocks of type $j$ of $\overline{\Pi}^{N,\veps}(n).$ Letting $\mathbf{Z}_{N,\veps}(n)=\big(Z^{(1)}_{N,\veps}(n),\ldots,Z^{(K)}_{N,\veps}(n)\big),$ we have defined a multi-type Galton-Watson process, of which we call $\mathbf{f}_{N,\veps}$ the generating function and its probabilities of extinction are $\mathbf{q}_{N,\veps}=(q^{(1)}_{N,\veps},\ldots,q^{(K)}_{N,\veps})$. One easily sees that $\mathbf{f}_{N,1/N}$ is nonincreasing in $N$ and converges to $\mathbf{f}$ on $[0,1]^K$, where $\mathbf{f}$ is the generating function corresponding to the non-truncated process, as in the proof of Lemma \ref{extinction}. By compactness, this convergence is in fact uniform on $[0,1]^K$.

Assume supercriticality for $(\mathbf{Z}(n),n\in\N)$ (otherwise the lemma is empty). This implies supercriticality of $(\mathbf{Z}_{N,1/N}(n),n\in\N)$ for $N$ large enough. Indeed, shortly, supercriticality means that the Perron eigenvalue of the matrix $\mathbf{M}=\big(\E_i[Z(1)^{(j)}]\big)_{i,j\in[K]}$ is strictly greater than $1$, and by monotonicity and continuity of this eigenvalue (Proposition \ref{algebre}), this will also be true for $\mathbf{M}_{N,\veps}=\big(\E_i[Z(1)_{N,\veps}^{(j)}]\big)_{i,j\in[K]}.$ This implies $q^{(i)}_{N,1/N}<1$ for $N$ large enough, and since the sequence is non-increasing, it stays bounded away from $1$. Taking the limit in the relation $\mathbf{f}_{N,1/N}(\mathbf{q}_{N,1/N})=\mathbf{q}_{N,1/N}$ then yields that, for any subsequential limit $\mathbf{q}'$, we have $\mathbf{f}(\mathbf{q}')=\mathbf{q'}$, and thus $\mathbf{q}'=\mathbf{q}$ by \cite[Corollary 1 of Theorem 7.2]{Ha}.
\end{proof}

\appendix

\section{Proofs of examples of the Malthusian hypothesis} \label{Appendix}
\subsection{Proof of Example \ref{example0}}
Notice that, for $i\in[K],$
\[\Big(\mathbf{\Phi}(q-1)\mathbf{1}\Big)_i=c_iq + \int_{\s}\big(1-\sum_{n\in\N} s_n^q\big) \mathrm d\nu_i(\bar{\mathbf{s}})=0.\]
Thus $\mathbf{1}$ is an eigenvector of $\mathbf{\Phi}(q-1)$ for the eigenvalue $0$. By Proposition \ref{algebre}, point $(ii)$, this implies $\lambda(q)=0.$
\qed
\subsection{Proof of Example \ref{example}}
Let $p\in[0,1]$. For $i\in[K]$, let 
\[f_i(p)=\int_{\s} \sum_{n=1}^{N}s_n^{p}\mathrm d\nu_i(\bar{\mathbf{s}}).\]
By assumption, $f_i$ is continuous and nonincreasing, and we have $f_i(1)\leq 1\leq f_i(0)$. In fact, by our non-degeneracy assumption at the start of Section \ref{fragbasics}, there is at least one $i$ such that $f_i$ is strictly decreasing. Also by assumption, we have, for $i,j\in[K]$:
\[\Big(\mathbf{\Phi}(p-1)\Big)_{i,j}=\begin{cases} 1 &\text{ if } j=i \\
                                         -f_i(p) &\text{ if } j=i+1 \\
                                         0 &\text{ otherwise}
                                        \end{cases}\]
Studying $\mathbf{\Phi}(p-1)$ is then straightforward: $(\mathbf{I}-\mathbf{\Phi}(p-1))^K=\big(\prod_{i=1}^K f_i(p)\big)\mathbf{I},$ which implies that $\mathbf{\Phi}(p-1)$ is diagonalisable and 
\[\lambda(p)=1-\big(\prod_{i=1}^K f_i(p)\big)^{1/K}.\]
One then readily obtains $\lambda(0)\leq 0 \leq \lambda(1)$, and thus there exists $p^*$ such that $\lambda(p^*)=0$ by the intermediate value theorem. More precisely, if $p> \max p^*_i$, then $f_i(p)\leq 1$ for all $i$, and the inequality is strict for at least one $i$, which implies $p^*\leq \max p^*_i$. A similar argument shows that $p^*\geq \min p^*_i$.
\qed
%**********************************************************************

\section*{Acknowledgments} Most of this work was completed while the author was a Post-Doctoral Fellow at NYU Shanghai. The author would like to thank Samuel Baumard, J\'er\^ome Casse and Raoul Normand for some stimulating discussions about this paper, and B\'en\'edicte Haas for some welcome advice and suggestions.
\bibliographystyle{siam}
\addcontentsline{toc}{section}{References}
\bibliography{frag}

\end{document}